\documentclass[11pt]{article}
\makeatother

\normalsize
\usepackage{amsfonts}
\usepackage{amsmath}
\usepackage{amssymb}
\usepackage{epstopdf}
\usepackage{amsthm}
\usepackage{epsfig}
\usepackage[mathscr]{euscript}
\usepackage[numbers,sort&compress]{natbib}
\usepackage{subfigure}
\allowdisplaybreaks

\newtheorem{theorem}{Theorem}

\newcommand{\be}{\begin{equation}}
\newcommand{\ee}{\end{equation}}

\begin{document}

%

\title{\bf Bogdanov-Takens bifurcation of codimension $3$ in the Gierer-Meinhardt model }
\author{\small Ranchao Wu\footnote{Corresponding author. E-mail address: rcwu@ahu.edu.cn (R. C. Wu)} and Lingling Yang\\
\baselineskip 12pt
\centerline{\small {\it Center for Pure Mathematics and School of Mathematical Sciences,}}\\
\centerline{\small{\it Anhui University, Hefei 230601, China }}}
\vskip.15in
\date{}

\maketitle

\begin{quote}
\noindent {\bf Abstract.}
Bifurcation of the local Gierer-Meinhardt model is analyzed in this paper. It is found that the degenerate Bogdanov-Takens bifurcation
of codimension 3 happens in the model, except that teh saddle-node bifurcation and the Hopf bifurcation. That was not reported in the
existing results about this model. The existence of equilibria, their stability,  the  bifurcation and the induced complicated and
interesting dynamics are explored in detail, by using the stability analysis, the normal form method and bifurcation theory.
Numerical results are also presented to  validate theoretical results.

\noindent {\bf Key words:}
Gierer-Meinhardt model, Saddle-node, Hopf, Bogdanov-Takens bifurcation

\end{quote}

\section{Introduction}\label{section-1}
Early in \cite{Turing}, Turing  discovered the common properties of the breakdown of spatial-temporal symmetry and the self-organization,
selection, and stability of new spatial-temporal structures in systems, and proposed the idea of patterns as the results
of diffusion driven instability.  Since then  more and more interests are focused on the Turing patterns and various models are put forward to
describe the diffusion driven instability. One of the important models is the Gierer-Meinhardt model \cite{GM}, which was proposed by Gierer and Meinhardt
in 1972, and takes the  following form
\begin{equation*}
\begin{cases}
\frac{\partial a}{\partial t} =
\rho_0\rho+c\rho\frac{a^r}{h^s} -ua+D_a \frac{\partial^2 a}{\partial x^{2} },
\\
\frac{\partial h}{\partial t} =c'\rho'\frac{a^T}{h^u}-vh+D_h \frac{\partial^2 h}{\partial x^{2} }.
\end{cases}
\end{equation*}
where $a(x,t)$ and $h(x,t)$ respectively represent the concentration of activators and inhibitors at spatial position $x$ and time $t > 0$.
$\rho_0\rho$ and $\rho'$ are the source concentration of $a(x,t)$ and $h(x,t)$, respectively. The first-order kinetics of activator and
inhibitor are represented by $u_a$ and $v_h$, respectively. $D_a$ and $D_h$ represent the diffusion coefficients of activators and inhibitors,
respectively. Generally, it is necessary to assume $\frac{sT}{u+1}>r-1>0$, that is, $r\ge2(r\in\mathbb{Z} )$.

In view of Turing's idea about pattern formation, to explore the patterns in such model, it is very necessary to carry out the stability and instability
analysis. Instability will be accompanied by bifurcation in the model. Then spatiotemporal patterns will follow from the different bifurcation.
Until not, various results about bifurcation and the resulting complex dynamics in the Gierer-Meinhardt model have been obtained.
When $r=2, s=1, T=2$, and $u=0$, Song et al. \cite{S} studied the Gierer-Meinhardt model with saturation terms and obtained the pattern formation
in the certain parameter space.  The Hopf bifurcation, the effect of diffusion on the stability and the subsequent Turing pattern were presented in
 \cite{W}. For the delayed a delayed reaction-diffusion Gierer-Meinhardt system, the bifurcation analysis was also carried out in \cite{YR}. With the
 different sources for  activators and inhibitors, Hopf bifurcation was treated in \cite{Rasoul}. For the codimension-2 bifurcation, in \cite{SONG}
the Turing-Hopf bifurcation was considered, without the saturation term. The Turing-Turing bifurcation was given in \cite{Zhao}, the coexistence of
 multi-stable patterns with different spatial responses and the superposition for patterns were demonstrated.

Recently, some results are obtained about the localized patterns in the Gray-Scott system and the bifurcation of the general
Gierer-Meinhardt model in \cite{F}. The local one-dimensional Gierer-Meinhardt model was given by
\begin{equation}\label{system1.1}
\begin{cases}
\frac{\partial u}{\partial t}= a+\frac{ u^{2} }{v}-u,\\
\frac{\partial v}{\partial t}=b+u^{2}-dv.
\end{cases}
\end{equation}
where $a$, $b$ and $d$  are all positive constants. However, when $a=0$, the system still has more complex dynamics and could be further explored.
In this work, it is found that the model could admit the saddle-node, the Hopf and the degenerate Bogdanov-Takes bifurcations of codimension-3,
which is not absent in the system in  \cite{F}. Note that highly degenerate bifurcations are more difficult to deal with and the resulting dynamical behaviors are richer and more interesting, so they are attracting increasing interests from mathematics and applied sciences. For example,
degenerate bifurcations and the induced complicated dynamics were presented in\cite{Etoua,Shang,AA}, such as the nilpotent cusp singularity of
order 3 and the degenerate Hopf bifurcation of codimension 3. In \cite{HXZR}, Huang et al. discovered that there existed a degenerate
Bogdanov-Takens singularity (focus case) of codimension 3 in the predator-prey model. In \cite{SLH},  the Bogdanov-Takens of codimension 3 and
the Hopf bifurcation of codimension 2 were also found to happen.

In this paper, we will elaborate on these aspects for system (\ref{system1.1}) with $a=0$.
The existence and their stability of  equilibrium points are introduced in Section 2. Bifurcations, such as, the saddle-node bifurcation,
the Hopf bifurcation and the Bogdanov-Takes bifurcation of codimension-3 are presented in Section 4. Finally, a brief summary is made in Section 5.

\section{Existence and stability of equilibria}\label{section-2}

Now consider the system  in the following form
\begin{equation}\label{system2.1}
\begin{cases}
	\frac{d u}{d t}= c \left ( \frac{\beta  u^{2} }{v}-u  \right ),\\
\frac{d v}{d t}=b+u^{2}-dv.
\end{cases}
\end{equation}
Let $f(u,v)=c \left ( \frac{\beta  u^{2} }{v}-u  \right )$,
$g(u,v)=b+u^{2}-dv$. Upon solving $f(u,v)=0$, we obtain the solutions $u=0$ or $v=u\beta$.

It is not difficult to  get the boundary equilibrium $(0,\frac{b}{d})$ of system (\ref{system2.1}).
Next, to find the existence of positive equilibria of system, substitute $v=u\beta$ into $g(u,v)=0$, then we have
\begin{equation*}
h(u) \triangleq u^{2}-d\beta u+b=0.
\end{equation*}
The discriminant of $h(u)$ is
\begin{equation*}
\Delta =d^{2} \beta ^{2} -4b.
\end{equation*}
It follows that

(i) if $d^{2} \beta ^{2} <4b$, then $h(u)>0$ for $u>0$;

(ii) if $d^{2} \beta ^{2} =4b$, then $h(u)$ has a real root $u_{1} =\frac{d\beta }{2}$;

(iii) if $d^{2} \beta ^{2} >4b$, then $h(u)$ has two distinct  positive real roots,
\begin{align*}
u_{2} =\frac{d\beta +\sqrt{\Delta} }{2 }, \qquad u_{3} =\frac{d\beta -\sqrt{\Delta} }{2 } .
\end{align*}

So we have the following result.
\begin{theorem}
 System (\ref{system2.1}) has only one boundary equilibrium
$E_{0}\left ( 0,v_{0}  \right )=\left (0,\frac{b}{d}  \right ) $, and

\;\;\;\;(i) if $d^{2} \beta ^{2} <4b$, then there is no positive equilibria;

\;\;\;\;(ii) If $d^{2} \beta ^{2} =4b$, then there is a positive equilibrium
$E_{1}\left ( u_{1},v_{1}  \right )=\left (\frac{d\beta}{2} ,\frac{d\beta^{2} }{2} \right ) $;

\;\;\;\;(iii) If $d^{2} \beta ^{2} >4b$,then there are two positive equilibrium
$E_{2}\left ( u_{2},v_{2}  \right )=\left (\frac{d\beta +\sqrt{\Delta} }{2 },\frac{d\beta^{2}  +\beta \sqrt{\Delta} }{2 } \right ) $
and $E_{3}\left ( u_{3},v_{3}  \right )=\left (\frac{d\beta -\sqrt{\Delta} }{2 } ,
\frac{d\beta^{2}  -\beta \sqrt{\Delta} }{2 } \right ) $.
\end{theorem}

Next the stability of the equilibria   system (\ref{system2.1}) will be examined. First consider the boundary equilibrium
$E_{0}\left ( 0,v_{0}  \right )$. The Jacobian matrix of system \ (\ref{system2.1}) at equilibirum $E_{0}$ is
\begin{equation*}\
	J_{E_{0} }=\begin{pmatrix}-1
		&0 \\
		0 &-d
	\end{pmatrix} ,
\end{equation*}
which has the eigenvalues $\lambda _{1}=-1<0$ and $\lambda _{2}=-d<0$.
Therefore, the equilibrium $E_{0}$ of system \ (\ref{system2.1}) is a stable node.

As for the stability of the equilibrium $E_{1}$, we have
\begin{theorem}

(a) If $d=c$, then $E_{1}$ is a cusp of codimension three;

(b) If $d>c$, then $E_{1}$ is a saddle-node with an unstable parabolic sector;

(b) If $d<c$, then $E_{1}$ is a saddle-node with a stable parabolic sector.
\end{theorem}

\begin{proof}
The Jacobian matrix of system (\ref{system2.1}) at equilibrium $E_{1}$ is
\begin{equation*}
	J_{E_{1} }=\begin{pmatrix}c
	&-\frac{c}{\beta } \\
	d\beta &-d
\end{pmatrix} .
\end{equation*}
It follows that
\begin{equation*}
	tr J_{E_{1} }=c-d, \;\;det J_{E_{1} }=0.
\end{equation*}
Now  translate $E_{1}(u_{1},v_{1})=(\frac{d\beta}{2},\frac{d\beta^{2}}{2})$
into the origin by the translation $(u,v)=(U+u_{1},V+v_{1})$, then
system (\ref{system2.1}) is changed into
\begin{equation}\label{system3.1}
	\begin{cases}	
		\dot U=a_{10}U+a_{01}V+a_{20}U^{2} +a_{11}UV+a_{02}V^{2}+a_{21}U^{2}V\\
	\qquad	+a_{12}UV^{2}+a_{03}V^{3}+a_{22}U^{2}V^{2}+a_{13}UV^{3}+a_{04}V^{4}
+M \left ( U,V \right ), \\
		\dot V=b_{10}U+b_{01}V+b_{20}U^{2}+N \left ( U,V \right ),
	\end{cases}	
\end{equation}
where
\begin{equation*}
\begin{aligned}
a_{10}&=c,&  a_{01}&=-\frac{c}{\beta},&  a_{20}&=\frac{2c}{d\beta},&
a_{11}&=-\frac{4c}{d\beta^{2}},&  a_{02}&=\frac{2c}{d\beta^{3}},\\
a_{21}&=-\frac{4c}{d^{2}\beta^{3}},&
a_{12}&=\frac{8c}{d^{2}\beta^{4}},&  a_{03}&=-\frac{4c}{d^{2}\beta^{5}},&
a_{22}&=\frac{8c}{d^{3}\beta^{5}},& a_{13}&=-\frac{16c}{d^{3}\beta^{6}},\\
a_{04}&=\frac{8c}{d^{3}\beta{7}},&
b_{10}&=d\beta,& b_{01}&=-d,& b_{20}&=1,
\end{aligned}
\end{equation*}
and $M(U,V)$, $N(U,V)$ are terms of at least order five in $U$ and $V$.

First, assume $d=c$. Then both eigenvalues of $J_{E_{1}}$ are zero.
Applying the transformation $(U,V)=(x,\beta(x-\frac{y}{c}))$, we rewrite system (\ref{system3.1}) as
\begin{equation}\label{system3.2}
	\begin{cases}
		\dot{x}  = y+\frac{2y^{2} }{c^{2}\beta}-\frac{4xy^{2}}{c^{3}\beta ^{2} }
		+\frac{4y^{3}}{c^{4}\beta^{2}}+\frac{8x^{2}y^{2}}{c^{4}\beta^{3}}-\frac{16xy^{3}}{c^{5}\beta^{3}}
		+\frac{8y^{4}}{c^{6}\beta^{3}}+M_{2} \left ( x,y \right ),  \\
		\dot{y}= -\frac{c}{\beta }x^{2}+\frac{2y^{2} }{c\beta }-\frac{4xy^{2} }{c^{2}\beta^{2}}
		+\frac{4y^{3}}{c^{3}\beta ^{2}}+\frac{8x^{2}y^{2}}{c^{3}\beta ^{2}}-\frac{16xy^{3} }{c^{4} \beta ^{3} }
		+\frac{8y^{4} }{c^{5}\beta ^{3}  }+N_{2} \left ( x,y \right ),
	\end{cases}
\end{equation}
and $M_{2}(x,y)$, $N_{2}(x,y)$ are terms of at least order five in $x$ and $y$.
Further,  let $(x,y)=(x_{1},y_{1}+x_{1}^{2}+\frac{2}{c\beta } x_{1} y_{1} -\frac{2}{c^{2} \beta } y_{1}^{2} )$, then
(\ref{system3.2}) is transformed  into the following form
\begin{equation}\label{system3.3}
	\begin{cases}
		\dot x_{1}=y_{1}+x_{1}^{2}+c_{11}x_{1}y_{1}+c_{21}x_{1}^{2}y_{1}+c_{12}x_{1}y_{1}^{2}+c_{03}y_{1}^{3}\\
		\qquad+c_{40}x_{1}^{4}+c_{22}x_{1}^{2}y_{1}^{2}+c_{13}x_{1}y_{1}^{3}+c_{04}y_{1}^{4}+M_{3}(x_{1},y_{1}),         \\
		\dot y_{1}=d_{20}x_{1}^2+d_{11}x_{1}y_{1}+d_{30}x_{1}^{3}
		+d_{21}x_{1}^{2}y_{1}+d_{12}x_{1}y_{1}^{2}+d_{03}y_{1}^{3}\\
		\qquad+d_{40}x_{1}^{4}+d_{31}x_{1}^{3}y_{1}+d_{22}x_{1}^{2}y_{1}^{2}
		+d_{13}x_{1}y_{1}^{3}+d_{04}y_{1}^{4}+N_{3}(x_{1},y_{1}),
	\end{cases}
\end{equation}
where
\begin{equation*}
\begin{aligned}
c_{11}&=\frac{2}{c\beta},\quad c_{21}=\frac{4}{c^{2}\beta},\quad c_{12}=\frac{4}{c^{3}\beta^{2}},
\quad c_{03}=-\frac{4}{c^{4}\beta^{2}},\quad c_{40}=\frac{2}{c^{2}\beta}, \quad c_{22}=\frac{4}{c^{4}\beta^{2}},\\
c_{13}&=\frac{8}{c^{5}\beta^{3}},\quad c_{04}=-\frac{8}{c^{6}\beta^{3}},\quad d_{20}=-\frac{c}{\beta},\quad d_{11}=-2, \quad
d_{30}=-2+\frac{2}{\beta^{2}},\\
d_{21}&=-\frac{4}{c\beta^{2}}+\frac{2}{c\beta},\quad
d_{12}=-\frac{8}{c^{2}\beta},\qquad d_{03}=-\frac{4}{c^{3}\beta^{2}}, \quad
d_{40}=-\frac{6}{c\beta}-\frac{4}{c\beta^{3}},\\
d_{31}&=\frac{4}{c^{2}\beta^{2}}+\frac{16}{c^{2}\beta^{3}}-\frac{16}{c^{2}\beta},\quad
d_{22}=\frac{12}{c^{3}\beta^{2}}-\frac{16}{c^{3}\beta^{3}}, \quad
d_{13}=\frac{8}{c^{4}\beta^{3}}-\frac{24}{c^{4}\beta^{2}},\quad d_{04}=-\frac{16}{c^{5}\beta^{3}},
\end{aligned}
\end{equation*}
and $M_{3}(x_{1},y_{1})$, $N_{3}(x_{1},y_{1})$ are terms of at least order five in $x_{1}$ and $y_{1}$.

Let $(x_{2},y_{2})=(x_{1},y_{1}+x_{1}^{2}+\frac{2}{c\beta } x_{1} y_{1}+M_{4}(x_{2},y_{2}))$, then (\ref{system3.3})
takes  the following form
\begin{equation}\label{system3.4}
	\begin{cases}
		\dot x_{2}=y_{2}, \\
		\dot y_{2}=e_{20}x_{2}^{2}+e_{02}y_{2}^{2}+e_{21}x_{2}^{2}y_{2}+e_{12}x_{2}y_{2}^{2}+e_{03}y_{2}^{3}\\
		\qquad+e_{40}x_{2}^{4}+e_{31}x_{2}^{3}y_{2}+e_{22}x_{2}^{2}y_{2}^{2}
		+e_{13}x_{2}y_{2}^{3}+e_{04}y_{2}^{4}+N_{4}(x_{2},y_{2}),
	\end{cases}
\end{equation}
where
\begin{equation*}
\begin{aligned}
e_{20}&=-\frac{c}{\beta},\quad e_{02}=\frac{2}{c\beta},\quad e_{21}=-\frac{4}{c\beta^{2}},\quad
e_{12}=-\frac{4}{c^{2}\beta^{2}}-\frac{8}{c^{2}\beta},\\
 e_{03}&=-\frac{4}{c^{3}\beta^{2}}, \quad
e_{40}=\frac{4}{c\beta^{3}},\quad
e_{31}=\frac{8}{c^{2}\beta^{2}}+\frac{16}{c^{3}\beta^{3}}+\frac{16}{c^{2}\beta^{3}},\quad
e_{22}=-\frac{8}{c^{3}\beta^{3}}+\frac{40}{c^{3}\beta^{2}}-\frac{16}{c^{4}\beta^{4}},\\
e_{13}&=-\frac{24}{c^{4}\beta^{2}}+\frac{24}{c^{4}\beta^{3}},\quad e_{04}=-\frac{16}{c^{5}\beta^{3}},
\end{aligned}
\end{equation*}
and $M_{4}(x_{2},y_{2})$, $N_{4}(x_{2},y_{2})$ are terms of at least order five in $x_{2}$ and $y_{2}$.

To eliminate the $y_{2}-$ term in (\ref{system3.4}), change system (\ref{system3.4})
with the following transformation \cite{SLH}
\begin{equation*}
	\begin{cases}
		x_{3}=x_{2}-\frac{e_{02}}{2}x_{2}^{2}-\frac{e_{21}}{3e_{20}}x_{2}y_{2}
		-\frac{e_{12}-e_{02}^{2}}{6}x_{2}^{3}-\frac{e_{03}e_{20}-e_{02}e_{21}}{2e_{20}}x_{2}^{2}y_{2}\\
		\qquad
		-\frac{9e_{02}^{3}e_{20}-27e_{12}e_{02}e_{20}+18e_{20}e_{22}-32e_{21}^{2}}{216e_{20}}x_{2}^{4}
		-\frac{7e_{02}^{2}e_{21}-12e_{02}e_{03}e_{20}-4e_{12}e_{21}+3e_{13}e_{20}}{18e_{20}}x_{2}^{3}y_{2}\\
		\qquad+\frac{e_{03}e_{21}-e_{04}e_{20}}{2e_{20}}x_{2}^{2}y_{2}^{2},\\
		y_{3}=y_{2}-e_{02}x_{2}y_{2}-\frac{e_{21}}{3e_{20}}y_{2}^{2}-\frac{e_{21}}{3}x_{2}^{3}
		-\frac{e_{12}-e_{02}^{2}}{2}x_{2}^{2}y_{2}-\frac{-2e_{02}e_{21}+3e_{03}e_{20}}{3e_{20}}x_{2}y_{2}^{2}\\
		\qquad
		-\frac{-3e_{02}e_{20}e_{21}+3e_{03}e_{20}^{2}+2e_{21}e_{30}}{6e_{20}}x_{2}^{4}
		-\frac{9e_{02}^{3}e_{20}-27e_{02}e_{12}e_{20}+18e_{20}e_{22}-14e_{21}^{2}}{54e_{20}}x_{2}^{3}y_{2}\\
		\qquad-\frac{4e_{20}^{2}e_{21}-9e_{02}e_{03}e_{20}-2e_{12}e_{21}+3e_{13}e_{20}}{6e_{20}}x_{2}^{2}y_{2}^{2}
		-\frac{-2e_{03}e_{21}+3e_{04}e_{20}}{3e_{20}}x_{2}y_{2}^{3},
	\end{cases}
\end{equation*}
then we get
\begin{equation}\label{system3.5}
	\begin{cases}
		\dot x_{3}=y_{3},\\
		\dot y_{3}=f_{20}x_{3}^{2}+f_{40}x_{3}^{4}+f_{31}x_{3}^{3}y_{3}+N_{5}(x_{3},y_{3}),
	\end{cases}
\end{equation}
where
\begin{align*}
f_{20}=-\frac{c}{\beta}, \;\;f_{40}=\frac{11}{3c\beta^{3}}-\frac{4}{3c\beta^{2}},\;\;
f_{31}=-\frac{4}{c^{2}\beta^{3}}+\frac{16}{c^{3}\beta^{3}}+\frac{8}{c^{2}\beta^{2}},
\end{align*}
and $M_{5}(x_{3},y_{3})$, $N_{5}(x_{3},y_{3})$ are terms of at least order five in $x_{3}$ and $y_{3}$.

Since $f_{20}=\frac{c}{\beta}\ne 0$, by the change of variables
$\left ( x_{4},y_{4} \right )=\left ( -x_{3},-\frac{1}{\sqrt{-f_{20}} }y_{3}  \right ),\tau =\sqrt{-f_{20}}t$,
we could turn system (\ref{system3.5}) into
\begin{equation}\label{system3.6}
	\begin{cases}
		\frac{\mathrm{d} x_{4}}{\mathrm{d} \tau} =y_{4},\\
		\frac{\mathrm{d} y_{4}}{\mathrm{d} \tau}
		= x_{4}^{2}+\left ( \frac{4}{3c^{2}\beta }-\frac{11}{3c^{2}\beta ^{2}}   \right )x_{4}^{3}
		-\frac{f_{31}}{\sqrt{-f_{20}} }x_{4}^{3}y_{4}+N_{6}(x_{4},y_{4}),
	\end{cases}
\end{equation}
where $N_{6}(x_{4},y_{4})$ are terms of at least order five in $x_{4}$ and $y_{4}$.

From the proposition $5.3$ in  \cite{L}, we know that system (\ref{system3.4}) is equivalent to the system
\begin{equation*}
	\begin{cases}
		\frac{\mathrm{d} x_{4}}{\mathrm{d} \tau} =y_{4},\\
		\frac{\mathrm{d} y_{4}}{\mathrm{d} \tau}  =x_{4}^{2}+Ex_{4}^{3}y_{4}+N_{6}(x_{4},y_{4}),
	\end{cases}
\end{equation*}
where $E=-\frac{f_{31}}{\sqrt{-f_{20}} }
=\frac{\frac{4}{c^{2}\beta^{3}}-\frac{16}{c^{3}\beta^{3}}-\frac{8}{c^{2}\beta^{2}}}{\sqrt{-f_{20}}} \ne 0$.
Therefore, $E_{1}$ is a cusp of codimension three. This proves $(a)$.

Next, assume $d \ne c$. The eigenvalues of $J_{E_{1}}$ are $\lambda _{3}=0$ and $\lambda _{4}=c-d$.
Applying the transformation
\begin{equation*}
	U=\frac{u}{\beta }+\frac{v}{\beta d}, V=u+v, \tau=\left ( d-c \right )t,
\end{equation*}
then  system (\ref{system3.1}) becomes
\begin{equation*}
	\begin{cases}
		\frac{\mathrm{d} u}{\mathrm{d} \tau}=-\frac{3}{\beta ^{2}\left ( d-c \right )^{2} }u^{2}
		-\frac{4d+2}{\beta ^{2}\left ( d-c \right )^{2}d }uv
		+\frac{1-4d}{\beta^{2}d^{2}\left ( d-c \right )^{2} } v^{2} +M_{7}(u,v),\\
		\frac{\mathrm{d} v}{\mathrm{d} \tau}=v-\frac{2+d}{\left ( d-c \right )^{2}\beta ^{2} }u^{2}
		-\frac{6}{\beta ^{2}\left ( d-c \right )^{2}}uv
		+\frac{2-5d}{\beta ^{2}d^{2}\left ( d-c \right )^{2} }v^{2} + N_{7}(u,v),
	\end{cases}
\end{equation*}
and $M_{7}(u,v),N_{7}(u,v)$ are terms of at least order three in $u$ and $v$.
The coefficient of $u^{2}$ is $-\frac{3}{\beta ^{2}\left ( d-c \right )^{2} }<0$.
From Theorem 7.1 in \cite{Zhang}, the origin is a saddle-node.
Considering the time vatiable $\tau$, if $d-c<0$, then $E_{1}$ is a saddle-node with a stable parabolic sector;
if $d-c>0$, then $E_{1}$ is a saddle-node with an unstable parabolic sector.
\end{proof}

If $\Delta>0$, then $h(u)$ has two equilibria. Finally, we discuss the stability of the positive equilibria $E_{2}$ and $E_{3}$.

\begin{theorem}
 If $\Delta>0$, then the positive equilibrium $E_{3}$ of system (\ref{system2.1})
is always a saddle point and the positive equilibrium $E_{2}$ is

(a) a source if $d<c$;

(b) a center or a fine focus if $d=c$;

(c) a sink if $d>c$.
\end{theorem}

\begin{proof}
The Jacobian matrix of system (\ref{system2.1}) at equilibrium $E_{2}$ and $E_{3}$ are
\begin{equation*}
	J_{E_{2}}=\begin{pmatrix}
		c& -\frac{c}{\beta } \\
		d\beta +\sqrt{\Delta }   &-d
	\end{pmatrix} ,\qquad
\ J_{E_{3}}=\begin{pmatrix}
	c& -\frac{c}{\beta } \\
	d\beta -\sqrt{\Delta }   &-d
\end{pmatrix}.
\end{equation*}
Then we could have
\begin{equation*}
	det J_{E_{2}}=\frac{c\sqrt{d^{2}\beta ^{2}-4b} }{\beta }>0
\end{equation*}
and
\begin{equation*}
	det J_{E_{3}}=-\frac{c\sqrt{d^{2}\beta ^{2}-4b} }{\beta }<0.
\end{equation*}
So $E_{3}$ is always a saddle point.
The positive equilibrium $E_{2}$ is determined by the sign of the trace $tr J_{E_{2}}$.
Specifically, when $tr J_{E_{2}}>0$, i.e., $d<c$,  $E_{2}$ is a source;
When $tr J_{E_{2}}<0$, i.e., $d>c$,  $E_{2}$ is a sink. when $tr J_{E_{2}}=0$,i.e., $d=c$, it is a center or a fine focus.
\end{proof}

\section{Bifurcation}\label{section-4}

\subsection{Saddle-node bifurcation}\label{section-4.1}

From Theorem 1 we note  that the equilibrium points of system (\ref{system2.1}) vary as the parameter $b$ changes.
When $b>\frac{d^{2} \beta ^{2}}{4}$, there is no positive equilibrium point. When $b=\frac{d^{2} \beta ^{2}}{4}$, there is a positive
equilibrium. When $b<\frac{d^{2} \beta ^{2}}{4}$, there are two positive equilibria. This indicates the saddle-node bifurcation may occur
around $E_1$.

\begin{theorem}
When $b=b_{SN}$, the system (\ref{system2.1}) undergoes the saddle-node bifurcation around $E_1$, with the
 threshold value $b_{SN}=\frac{d^{2} \beta ^{2}}{4}$.
\end{theorem}

\begin{proof}
According to the Sotomayor's theorem \cite{Perko}, we need to verify the transversality condition for the occurrence of
saddle-node bifurcation at $b\equiv b_{SN}$. The Jacobian matrix of system (\ref{system2.1}) at equilibrium $E_{1}$ is
\begin{equation*}
	J_{E_{1}}=\begin{pmatrix}
		c& -\frac{c}{\beta} \\
		d\beta &-d
	\end{pmatrix}.
\end{equation*}
Because of $det(J_{E_{1}})=\lambda_{5}\lambda_{6}=0$, $J_{E_{1}}$ has a zero eigenvalue $\lambda_{5}$.
Let $V$ and $W$ represent the eigenvectors of $J_{E_{1}}$ and $J_{E_{1}}^{T}$
with respect to the eigenvalue $\lambda_{5}$, respectively.

Simple calculation gives
\begin{equation*}
	V=\binom{V_{1}}{V_{2}}=\binom{1}{\beta}
\end{equation*}
and
\begin{equation*}
	W=\binom{W_{1}}{W_{2}}=\binom{1}{-\frac{c}{d\beta}}.
\end{equation*}
Further, we can obtain
\begin{equation*}
	F_{b}(E_{1},b_{SN})=\binom{0}{1},
\end{equation*}

\begin{equation*}
	D^{2}F(E_{1},b_{SN})(V,V)=\begin{pmatrix}
		\frac{\partial^{2} f}{\partial u^{2}}V_{1}^{2}
		+2 \frac{\partial^2 f}{\partial u\partial v} V_{1}V_{2}
		+\frac{\partial^{2} f}{\partial v^{2}}V_{2}^{2} \\
		\frac{\partial^{2} g}{\partial u^{2}}V_{1}^{2}
		+2 \frac{\partial^2 g}{\partial u\partial v} V_{1}V_{2}
		+\frac{\partial^{2} g}{\partial v^{2}}V_{2}^{2}
	\end{pmatrix}_{(E_{1},b_{SN})}=\binom{0}{2} .
\end{equation*}
Obviously, when $b\equiv b_{SN}$, $V$ and $W$ satisfy the transversality conditions
\begin{equation*}
	W^{T}F_{b}(E_{1},b_{SN}) =-\frac{c}{d\beta }\ne 0
\end{equation*}
and
\begin{equation*}
	W^{T}\left [ D^{2} F(E_{1},b_{SN})(V,V)  \right ] =-\frac{2c}{d\beta } \ne 0  .
\end{equation*}
Therefore, when the parameter $b$ goes from one side of $b_{SN}$ to the other, the system (\ref{system2.1}) experiences the saddle-node
bifurcation at the equilibrium point $E_{1}$.
\end{proof}

\subsection{Hopf bifurcation}\label{section-4.2}

From Theorem 3, it is found that the stability of the positive equilibrium of $E_{2}$ changes as the sign of
$tr(E_{2})$ varies, that will probably lead to  the Hopf bifurcation around $E_{2}$.

According to the Hopf bifurcation theorem, we need to verify the transversal condition.
Based on the fact that $\frac{\mathrm{d} tr\left ( J_{E_{2}} \right )}{\mathrm{d} d}\bigg|_{d = c} =-1\ne 0$,
system (\ref{system2.1}) undergoes the Hopf-bifurcation around $E_{2}$.
Furthermore, we need to give the first Lyapunov coefficient and determine the stability of the limit cycle around $E_{2}$.

Now translate $E_{2}(u_{2},v_{2})$ to $(0,0)$  by
$(\hat{u},\hat{v})=(u-u_{2},v-v_{2})$ and get the following form
\begin{equation}\label{system4.1}
	\begin{cases}
		\dot{\hat{u} } =c\hat{u}-\frac{c}{\beta } \hat{v} +\frac{c}{u_{2}}\hat{u}^{2 }
		-\frac{2c}{\beta u_{2}}\hat{u}\hat{v} +\frac{c}{\beta ^{2}u_{2}}\hat{v}^{2}+\hat{M}(\hat{u},\hat{v}),\\
		\dot{\hat{v} } =2u_{2}\hat{u}-\hat{v} +\hat{u}^{2 }+ \hat{N}(\hat{u},\hat{v}),
	\end{cases}
\end{equation}
where $\hat{M}(\hat{u},\hat{v})$ and $ \hat{N}(\hat{u},\hat{v})$ are terms of at least order three in $\hat{u}$ and $\hat{v}$.

From another transformation
$\hat{x}=-\hat{u},\hat{y}=-\hat{u},\hat{y}=\frac{1}{\sqrt{D}}(c\hat{u}-\frac{c}{\beta } \hat{v} )$
and $\mathrm{d}\tau =\sqrt{D} \mathrm{d}t $, where $D=\frac{\sqrt{(d\beta)^{2}-4b} }{\beta }$,
system (\ref{system4.1}) becomes
\begin{equation*}
	\begin{cases}
		\dot{\hat{x}}=-\dot{\hat{y}}+f(\hat{x},\hat{y}  ), \\
		\dot{\hat{y}}=\dot{\hat{x}}+g(\hat{x},\hat{y}  ),
	\end{cases}
\end{equation*}
where
\begin{align*}
f(\hat{x},\hat{y} )=-\frac{\sqrt{D} }{u_{2}}\hat{y}^{2} +\hat{M_{1}}(\hat{u},\hat{v}),\quad
g(\hat{x},\hat{y}  )=-\frac{1 }{D\beta }\hat{x}^{2} +\frac{\hat{y}^{2}}{u_{2}}+\hat{N_{1}}(\hat{u},\hat{v}),
\end{align*}
where $\hat{M_1}(\hat{u},\hat{v})$ and $ \hat{N_1}(\hat{u},\hat{v})$ are also terms of at least order three in $\hat{u}$ and $\hat{v}$.

Using the formal series method described in \cite{Zhang}, we can calculate that the first-order Lyapunov number is
\begin{equation*}
	\sigma =\frac{\sqrt{D} }{4u_{2}^{2}} <0.
\end{equation*}

Then the following theorem is available.

\begin{theorem}
 When $\Delta>0$ and $d=c$, the system (\ref{system2.1}) at the equilibrium experiences the supercritical Hopf bifurcation with a stable limit cycle
 around $E_{2}$.
\end{theorem}

\subsection{Bogdanov-Takens bifurcation}\label{section-4.3}

When $u =\frac{v}{\beta}$ and $d=c$, it follows from Theorem 2(a) that the unique positive equilibrium $E_{1}$ of system (\ref{system2.1})
is a cusp of codimension three. The Bogdanov-Takens bifurcation may occur arround $E_{1}$.
Now we select $\beta$, $b$ and $d$ as bifurcation parameters, and the Bogdanov-Takens bifurcation may occur under parameter perturbation.

\begin{theorem}
 When $u =\frac{v}{\beta}$ and $d=c$, the parameter $(\beta,b,d)$ varies within the small neighborhood of $(\beta_{BT},b_{BT},d_{BT})$,
where $\beta_{BT},b_{BT},$ and $d_{BT}$ are the Bogdanov-Takens bifurcation threshold values.
Then system (\ref{system2.1}) undergoes the Bogdanov-Taken bifurcation of codimension 3 in the small neighborhood of $E_{1}$.
\end{theorem}

\begin{proof}
When $u =\frac{v}{\beta}$ and $d=c$, it follows theorem 2(a) that $E_{1}$ is a cusp of codimension three of system (\ref{system2.1}).
Perturb the parameters $\beta$, $b$ and $d$  at $\beta_{BT}$, $b_{BT}$ and $d_{BT}$ and denote
$(\beta,b,d)=(\beta_{BT}+\epsilon _{1},b_{BT}+\epsilon _{2}, d_{BT}+\epsilon _{3})$,
where $\epsilon=(\epsilon_{1},\epsilon_{2},\epsilon_{3})$ is a vector of parameters in the small neighborhood of $(0,0,0)$.
Then, the system (\ref{system2.1}) becomes
\begin{equation}\label{system4.2}
	\begin{cases}
		\frac{\mathrm{d} u}{\mathrm{d} t}=c(\frac{(\beta+\epsilon _{1})u^{2}}{v}-u ) ,\\
		\frac{\mathrm{d} v}{\mathrm{d} t}=b+\epsilon _{2}+u^{2}-(d+\epsilon _{3})v.
	\end{cases}
\end{equation}
Then we translate $E_{1}(u_{1},v_{1})=(\frac{c\beta}{2},\frac{c\beta^{2}}{2})$ into the origin by
$(x,y)=(u-u_{1},v-v_{1})$. The system (\ref{system4.2}) is changed into
\begin{equation}\label{system4.3}
	\begin{cases}
		\dot x= \bar{a}_{00} +\bar{a}_{10}x+\bar{a}_{01}y+\bar{a}_{20}x^{2}+\bar{a}_{11}xy
		+\bar{a}_{02}y^{2}+\bar{a}_{21}x^{2}y+\bar{a}_{12}xy^{2}  \\
		\qquad+\bar{a}_{03}y^{3}+\bar{a}_{22}x^{2}y^{2}+\bar{a}_{13}xy^{3}
		+\bar{a}_{04}y^{4}+O(|x,y|^{4}), \\
		\dot y= \bar{b}_{00}+ \bar{b}_{10}x+ \bar{b}_{01}y+ \bar{b}_{20}x^{2}+ O(|x,y|^{5}),
	\end{cases}
\end{equation}
where
\begin{align*}
\bar{a}_{00}&=\frac{c^{2}\epsilon_{1}}{2},\quad
\bar{a}_{10}=c(\frac{2(\beta+\epsilon_{1})}{\beta}-1),\quad
\bar{a}_{01}=-\frac{c(\beta+\epsilon_{1})}{\beta^{2}},\quad
\bar{a}_{20}=\frac{2(\beta+\epsilon_{1})}{\beta^{2}},\\
\bar{a}_{11}&=-\frac{4(\beta+\epsilon_{1})}{\beta^{3}},\quad
\bar{a}_{02}=\frac{2(\beta+\epsilon_{1})}{\beta^{4}},\quad
\bar{a}_{21}=-\frac{4(\beta+\epsilon_{1})}{c\beta^{4}},\quad
\bar{a}_{12}=\frac{8(\beta+\epsilon_{1})}{c\beta^{5}},\\
\bar{a}_{03}&=-\frac{4(\beta+\epsilon_{1})}{c\beta^{6}},\quad
\bar{a}_{22}=\frac{8(\beta+\epsilon_{1})}{c^{2}\beta^{6}},\quad
\bar{a}_{13}=-\frac{16(\beta+\epsilon_{1})}{c^{2}\beta^{7}},\quad
\bar{a}_{04}=\frac{8(\beta+\epsilon_{1})}{c^{2}\beta^{8}},\\
\bar{b}_{00}&=b+\epsilon_{2}-\frac{1}{4}c^{2}\beta^{2}-\frac{c\beta^{2}\epsilon_{3}}{2},\quad
\bar{b}_{10}=c\beta,\quad
\bar{b}_{01}=-d-\epsilon_{3},\qquad \bar{b}_{20}=1.
\end{align*}

Then we rewrite system (\ref{system4.3}) with the transformation
$(x,y)=(x_{1}-\frac{2}{c\beta^{2}}x_{1}y_{1},y_{1})$ to
\begin{equation*}
	\begin{cases}
		\dot x_{1} =\bar{c}_{00}+\bar{c}_{10}x_{1}+\bar{c}_{01}y_{1}
		+\bar{c}_{20}x_{1}^{2}+\bar{c}_{11}x_{1}y_{1}+ O(|x,y|^{4}),
		\\
		\dot y_{1} =\bar{d}_{00}+\bar{d}_{10}x_{1}+\bar{d}_{01}y_{1}
		+\bar{d}_{20}x_{1}^{2}+\bar{d}_{11}x_{1}y_{1}+ O(|x,y|^{4}) ,
	\end{cases}
\end{equation*}
where
\begin{align*}
\bar{c}_{00}&=\frac{c^{2}\epsilon_{1}}{2}\quad
\bar{c}_{10}=\frac{2}{c\beta^{2}}\bar{b}_{00}+c(\frac{2(\beta+\epsilon_{1})}{\beta}-1),\quad
\bar{c}_{01}=-\frac{c}{\beta},\\
\bar{c}_{20}&=\frac{4}{\beta}+\frac{2\epsilon_{1}}{\beta^{2}},\quad
\bar{c}_{11}=-\frac{2(c+\epsilon_{3})}{c\beta^{2}}-\frac{4(\beta+\epsilon_{1})}{\beta^{3}},\quad
\bar{d}_{00}=\bar{b}_{00},\\ \bar{d}_{10}&=c\beta,\quad
\bar{d}_{01}=-c-\epsilon_{3},\qquad\bar{d}_{20}=1,\quad
\bar{d}_{11}=-\frac{2}{\beta}.
\end{align*}

Further, we execute the transformation $(x_{2},y_{2})=(x_{1},\dot{x}_{1})$ and get
\begin{equation}\label{system4.4}
	\begin{cases}
		\dot x_{2} = y_{2},
		\\
		\dot y_{2} =\bar{e}_{00}+\bar{e}_{10}x_{2}+\bar{e}_{01}y_{2}+\bar{e}_{20}x_{2}^{2}
		+\bar{e}_{11}x_{2}y_{2}+ \bar{e}_{02}y_{2}^{2}+\bar{e}_{30}x_{2}^{3}
		+\bar{e}_{21}x_{2}^{2}y_{2}+\bar{e}_{12}x_{2}y_{2}^{2}\\
		\qquad+\bar{e}_{03}y_{2}^{3}+\bar{e}_{40}x_{2}^{4}+\bar{e}_{31}x_{2}^{3}y_{2}
		+\bar{e}_{22}x_{2}^{2}y_{2}^{2}+\bar{e}_{13}x_{2}y_{2}^{3}+\bar{e}_{04}y_{2}^{4}+ O(|x,y|^{4}),
	\end{cases}
\end{equation}
where
\begin{align*}
\bar{e}_{00}&=\bar{c}_{01}\bar{d}_{00}-\bar{c}_{00}\bar{d}_{01},\quad
\bar{e}_{10}=\bar{c}_{11}\bar{d}_{00}+\bar{c}_{01}\bar{d}_{10}-\bar{c}_{00}\bar{d}_{11}
-\bar{c}_{10}\bar{d}_{01},\quad
\bar{e}_{01}=\bar{c}_{10}+\bar{d}_{01}-\frac{\bar{c}_{00}\bar{c}_{11}}{\bar{c}_{01}},\\
\bar{e}_{20}&=\bar{c}_{11}\bar{d}_{10}+\bar{c}_{01}\bar{d}_{20}-\bar{c}_{10}\bar{d}_{11}-\bar{c}_{02}\bar{d}_{01},
\qquad \bar{e}_{11}=2\bar{c}_{20}-\frac{\bar{c}_{10}\bar{c}_{11}}{\bar{c}_{01}}
+\frac{\bar{c}_{00}\bar{c}_{11}^{2}}{\bar{c}_{01}^{2}}+\bar{d}_{11},\quad
\bar{e}_{02}=\frac{\bar{c}_{11}}{\bar{c}_{01}},\\
\bar{e}_{30}&=\bar{c}_{11}\bar{d}_{20}-\bar{c}_{02}\bar{d}_{11},\quad
\bar{e}_{21}=-\frac{\bar{c}_{11}\bar{c}_{02}}{\bar{c}_{01}}+\frac{\bar{c}_{11}^{2}\bar{c}_{10}}{\bar{c}_{01}^{2}}
-\frac{\bar{c}_{00}\bar{c}_{11}^{3}}{\bar{c}_{01}^{3}},\quad
\bar{e}_{12}=-\frac{\bar{c}_{11}^{2}}{\bar{c}_{01}^{2}},\quad
\bar{e}_{03}=0,\\
\bar{e}_{40}&=\frac{\bar{c}_{00}\bar{c}_{11}^{4}\bar{d}_{01}}{\bar{c}_{01}^{4}},\quad
\bar{e}_{31}=\frac{\bar{c}_{11}^{2}\bar{c}_{02}}{\bar{c}_{01}^{2}}
-\frac{\bar{c}_{10}\bar{c}_{11}^{3}}{\bar{c}_{01}^{3}}+\frac{\bar{c}_{00}\bar{c}_{11}^{4}}{\bar{c}_{01}^{4}},\quad
\bar{e}_{22}=\frac{\bar{c}_{11}^{3}}{\bar{c}_{01}^{3}},\quad
\bar{e}_{13}=0,\quad \bar{e}_{04}=0.
\end{align*}

To verify that the Bogdanov-Takens bifurcation  occurs at equilibrium point $E_{1}$, we need to get the universal unfolding of
system (\ref{system4.1}). So  we need to eliminate  $y_{2}^{2}$, $x^{3}$, $x^{2}y$, $xy^{2}$, $y^{3}$, and $x^{4}$ terms.
Next, we transform system (\ref{system4.4}) by the procedure similar to that in \cite{LCZ}.

(A) In order to eliminate the $y_{2}^{2}$ term, take the transformation
$(x_{2},y_{2})=(u_{1}+\frac{c_{02}}{2}u_{1}^{2},v_{1}+c_{02}u_{1}v_{1})$, then system (\ref{system4.4}) takes the following form
\begin{equation}\label{system4.5}
	\begin{cases}
		\frac{\mathrm{d} u_{1} }{\mathrm{d} t}  = v_{1},
		\\
		\frac{\mathrm{d} v_{1} }{\mathrm{d} t} =\bar{f}_{00}+\bar{f}_{10}u_{1}+\bar{f}_{01}v_{1}+\bar{f}_{20}u_{1}^{2}
		+\bar{f}_{11}u_{1}v_{1}+\bar{f}_{30}u_{1}^{3}+\bar{f}_{21}u_{1}^{2}v_{1}+\bar{f}_{12}u_{1}v_{1}^{2}\\
		\qquad+\bar{f}_{40}u_{1}^{4}+\bar{f}_{31}u_{1}^{3}v_{1}+\bar{f}_{22}u_{1}^{2}v_{1}^{2}+ O (|u_{1},v_{1}|^{5}) ,
	\end{cases}
\end{equation}
where
\begin{align*}
\bar{f}_{00}&=\bar{e}_{00},\quad
\bar{f}_{10}=\bar{e}_{10}-\bar{e}_{00}\bar{e}_{02},\quad
\bar{f}_{01}=\bar{e}_{01},\quad
\bar{f}_{20}=\bar{e}_{20}-\frac{\bar{e}_{10}\bar{e}_{02}}{2}+\bar{e}_{00}\bar{e}_{02}^{2},\\
\bar{f}_{11}&=\bar{e}_{11},\quad
\bar{f}_{30}=\bar{e}_{30}+\frac{\bar{e}_{10}\bar{e}_{02}^{2}}{2}-\bar{e}_{00}\bar{e}_{02}^{3},\quad
\bar{f}_{21}=\bar{e}_{21}+\frac{\bar{e}_{11}\bar{e}_{02}}{2},\\
\bar{f}_{40}&=\frac{\bar{e}_{20}\bar{e}_{02}^{2}}{4}-\frac{\bar{e}_{10}\bar{e}_{02}^{3}}{2}
+\bar{e}_{00}\bar{e}_{02}^{4}+\frac{\bar{e}_{02}\bar{e}_{30}}{2},\quad
\bar{f}_{31}=\bar{e}_{31}+\bar{e}_{02}\bar{e}_{21},\quad
\bar{f}_{22}=\bar{e}_{22}+\frac{3\bar{e}_{02}\bar{e}_{12}}{2}-\bar{e}_{02}^{3}.
\end{align*}

(B) In order to eliminate the $u_{1}v_{1}^{2}$ term, take the transformation
$(u_{1},v_{1})=(u_{2}+\frac{\bar{f}_{12}}{2}u_{2}^{3},v_{2}+\frac{d_{12}}{6}u_{2}^{2}v_{2})$,
then (\ref{system4.5}) is reduced to
\begin{equation}\label{system4.6}
	\begin{cases}
		\frac{\mathrm{d} u_{2} }{\mathrm{d} t}  = v_{2},
		\\
		\frac{\mathrm{d} v_{2} }{\mathrm{d} t}
		=\bar{g}_{00}+\bar{g}_{10}u_{2}+\bar{g}_{01}v_{2}+\bar{g}_{20}u_{2}^{2}
		+\bar{g}_{11}u_{2}v_{2}+\bar{g}_{30}u_{2}^{3}+\bar{g}_{21}u_{2}^{2}v_{2}\\
		\qquad+\bar{g}_{40}u_{2}^{4}+\bar{g}_{31}u_{2}^{3}v_{2}
		+\bar{g}_{22}u_{2}^{2}v_{2}^{2}+ O(|u_{2},v_{2}|^{5})  ,
	\end{cases}
\end{equation}
where
\begin{equation*}
\begin{aligned}
\bar{g}_{00}&=\bar{f}_{00},\quad
\bar{g}_{10}=\bar{f}_{10},\quad
\bar{g}_{01}=\bar{f}_{01},\quad
\bar{g}_{20}=\bar{f}_{20}-\frac{\bar{f}_{00}\bar{f}_{12}}{2}+\frac{\bar{f}_{10}\bar{f}_{12}}{6},\\
\bar{g}_{11}&=\bar{f}_{11},\quad
\bar{g}_{30}=\bar{f}_{30}-\frac{\bar{f}_{10}\bar{f}_{12}}{2}+\frac{\bar{f}_{20}\bar{f}_{12}}{3},\quad
\bar{g}_{21}=\bar{f}_{21}+\frac{\bar{f}_{11}\bar{f}_{12}}{6},\\
\bar{g}_{40}&=\bar{f}_{40}-\frac{\bar{d}_{20}\bar{d}_{12}}{6}+\frac{\bar{f}_{00}\bar{f}_{12}^{2}}{4},\quad
\bar{g}_{31}=\bar{f}_{31}+\frac{\bar{f}_{11}\bar{f}_{12}}{6},\quad
\bar{g}_{22}=\bar{f}_{22}.
\end{aligned}
\end{equation*}

(C) Notice that $\bar{g}_{20}=-\frac{2b}{c\beta^{3}}-\frac{9c}{2\beta}+O(\epsilon)\ne0$.
To removing the $u_{2}^{3}$ and $u_{2}^{4}$ terms, we transform system (\ref{system4.6}) with
$(u_{2},v_{2})=(u_{3}-\frac{\bar{g}_{30}}{4\bar{g}_{20}}u_{3}^{2}+\frac{15\bar{g}_{30}^{2}-16\bar{g}_{20}\bar{g}_{40}}
{80\bar{g}_{20}^{2}}u_{3}^{3},v_{3})$
and scaling transformation $\mathrm{d}\tau=(1+\frac{\bar{g}_{30}}{2\bar{g}_{20}}u_{3}
+\frac{48\bar{g}_{20}\bar{g}_{40}-25\bar{g}_{30}^{2}}{80\bar{g}_{20}}u_{3}^{2}
+\frac{48\bar{g}_{20}\bar{g}_{30}\bar{g}_{40}-35\bar{g}_{30}^{3}}{80\bar{g}_{20}^{3}}u_{3}^{3})\mathrm{d}t$ to obtain the following system
\begin{equation}
	\begin{cases}
		\frac{\mathrm{d} u_{3}}{\mathrm{d} \tau} =v_{3},\\
		\frac{\mathrm{d} v_{3}}{\mathrm{d} \tau }
		=\bar{i}_{00}+\bar{i}_{10}u_{3}+\bar{i}_{01}v_{3}+\bar{i}_{20}u_{3}^{2}
		+\bar{i}_{11}u_{3}v_{2}+\bar{i}_{21}u_{3}^{2}v_{3}+\bar{i}_{12}u_{3}v_{3}^{2}\\
		\qquad+\bar{i}_{40}u_{3}^{4}+\bar{i}_{31}u_{3}^{3}v_{3}+\bar{i}_{22}u_{3}^{2}v_{3}^{2}+ R(u_{3},v_{3},\epsilon) ,
	\end{cases}
\end{equation}
where
\begin{align*}
\bar{i}_{00}&=\bar{g}_{00},\quad
\bar{i}_{10}=\bar{g}_{10}-\frac{\bar{g}_{00}\bar{g}_{30}}{2\bar{g}_{20}},\quad
\bar{i}_{01}=\bar{g}_{01},\\
\bar{i}_{20}&=\bar{g}_{20}+\frac{45\bar{g}_{00}\bar{g}_{30}^{2}-60\bar{g}_{10}\bar{g}_{20}
               \bar{g}_{30}-48\bar{g}_{00}\bar{g}_{20}\bar{g}_{40}}{80\bar{g}_{20}^{2}},\\
\bar{i}_{11}&=\bar{g}_{11}-\frac{\bar{g}_{01}\bar{g}_{30}}{2\bar{g}_{20}},\quad
\bar{i}_{30}=\frac{\bar{g}_{10}(35\bar{g}_{30}^{2}-32\bar{g}_{20}\bar{g}_{40})}{4\bar{g}_{20}^{2}},\\
\bar{i}_{21}&=\bar{g}_{21}-\frac{60\bar{g}_{11}\bar{g}_{20}\bar{g}_{30}
             -45\bar{g}_{01}\bar{g}_{30}^{2}+48\bar{g}_{01}\bar{g}_{20}\bar{g}_{40}}{80\bar{g}_{20}^{2}},\\
\bar{i}_{40}&=\frac{g10g30g40}{4g20^2}-\frac{15g10g30^3}{64g20^3},\quad
\bar{i}_{31}=\bar{g}_{31}+\frac{7\bar{g}_{11}\bar{g}_{30}^{2}}{8\bar{g}_{20}^{2}}
-\frac{5\bar{g}_{30}\bar{g}_{21}+4\bar{g}_{11}\bar{g}_{40}}{5\bar{g}_{20}},\\ R(u_{3},v_{3},\epsilon)&=
v_3^2O(|u_3,v_3|^2)+O(|u_3,v_3|^5)+O(\epsilon )(O(v_3^2)+O(|u_3,v_3|^3))+O(\epsilon^2)O(|u_3,v_3|).
\end{align*}

(D) Since
\begin{eqnarray*}
\bar{i}_{20}&=&-\frac{9c}{2\beta}-\frac{2b}{c\beta^{3}}+\frac{2527200c^{3}}{\beta^{4}(\frac{c^{3}\beta}{4}-\frac{cb}{\beta})^{2}}
+\frac{62104320b^{2}}{c\beta^{7}(\frac{c^{3}\beta}{4}-\frac{cb}{\beta})^{2}}\\
&&-\frac{23034240cb}{\beta^{5}(\frac{c^{3}\beta}{4}-\frac{cb}{\beta})^{2}}
-\frac{416102406b^{3}}{c^{3}\beta^{9}(\frac{c^{3}\beta}{4}-\frac{cb}{\beta})^{2}}+O(\epsilon)\ne 0,
\end{eqnarray*}
by the transformation
$$(u_{3},v_{3})=(u_{4},v_{4}+\frac{\bar{i}_{21}}{3\bar{i}_{20}}v_{4}^{2}+\frac{\bar{i}_{21}^{2}}{36\bar{i}_{20}^{2}}v_{4}^{3}),
\mathrm{d}t=(1+\frac{\bar{i}_{21}}{3\bar{i}_{20}}v_{4}+\frac{\bar{i}_{21}^{2}}{36\bar{i}_{20}^{2}}v_{4}^{2})\mathrm{d}\tau,$$
we can obtain  the following form
\begin{equation}\label{system4.8}
	\begin{cases}
		\frac{\mathrm{d} u_{4}}{\mathrm{d} \tau} =v_{4},\\
		\frac{\mathrm{d} v_{4}}{\mathrm{d} \tau }
		=\bar{j}_{00}+\bar{j}_{10}u_{4}+\bar{j}_{01}v_{4}+\bar{j}_{20}u_{4}^{2}+\bar{j}_{11}u_{4}v_{4}
		+\bar{j}_{31}u_{4}^{3}v_{4}+R(u_{4},v_{4},\epsilon),
	\end{cases}
\end{equation}
where
\begin{align*}
\bar{j}_{00}&=\bar{i}_{00}, \quad \bar{j}_{10}=\bar{i}_{10},\quad
\bar{j}_{01}=\bar{i}_{01}-\frac{\bar{i}_{00}\bar{i}_{21}}{\bar{i}_{20}},\\
\bar{j}_{20}&=\bar{i}_{20},\quad
\bar{j}_{11}=\bar{i}_{11}-\frac{\bar{i}_{10}\bar{i}_{21}}{\bar{i}_{20}},\quad
\bar{j}_{31}=\bar{i}_{31}-\frac{\bar{i}_{21}\bar{i}_{30}}{\bar{i}_{20}}.
\end{align*}
Additionally, $R(u_{4},v_{4},\epsilon)$ has the same properties as $R(u_{3},v_{3},\epsilon)$.

(E) We have $\bar{j}_{20}$ and $\bar{j}_{31}$ with the help of MAPLE

\begin{equation*}
	\begin{split}
\bar{j}_{20}&=-\frac{2b}{c\beta^3}-\frac{9c}{2\beta}
	+\frac{2527200c^2}{\beta^3(\frac{\beta c^3}{4}-\frac{cb}{\beta})^2}
	+\frac{62104320b^2}{c\beta^7(\frac{\beta c^3}{4}-\frac{cb}{\beta})^2}
	-\frac{23034240cb}{\beta^5(\frac{\beta c^3}{4}-\frac{cb}{\beta})^2}
	-\frac{41610240b^3}{c^3\beta^9(\frac{\beta c^3}{4}-\frac{cb}{\beta})^2}+O(\epsilon)\\
&\ne0,\\
		\bar{j}_{31}&=-\frac{151200 b^3}{{\beta}^{11} c^6 \left(\frac{2 b {\beta}^3}{c}+\frac{9 {\beta} c}{2}\right)^2}
+\frac{128520 b^2}{{\beta}^9 c^4 \left(\frac{2 b {\beta}^3}{c}+\frac{9 {\beta} c}{2}\right)^2}+\frac{71136 b^2}{5 {\beta}^8 c^5
\left(\frac{2 b}{{\beta}^3 c}+\frac{9 c}{2 {\beta}}\right)}-(\frac{57816 b^3}{5 {\beta}^{10} c^5 \left(\frac{2 b {\beta}^3}{c}
+\frac{9 {\beta} c}{2}\right)^2}\\
		&-\frac{109278 b^2}{5 {\beta}^8 c^3 \left(\frac{2 b {\beta}^3}{c}+\frac{9 {\beta} c}{2}\right)^2}+\frac{36 b}{{\beta}^4 c^3}
+\frac{44361 b}{5 {\beta}^6 c \left(\frac{2 b {\beta}^3}{c}+\frac{9 {\beta} c}{2}\right)^2}-\frac{4131 c}{4 {\beta}^4
\left(\frac{2 b {\beta}^3}{c}+\frac{9 {\beta} c}{2}\right)^2}+\frac{27}{{\beta}^2 c})\\
		&\left(\frac{218976 b^3}{{\beta}^{10} c^4 \left(\frac{2 b {\beta}^3}{c}+\frac{9 {\beta} c}{2}\right)^2}
-\frac{273048 b^2}{{\beta}^8 c^2 \left(\frac{2 b {\beta}^3}{c}+\frac{9 {\beta} c}{2}\right)^2}+\frac{93456 b}{{\beta}^6
\left(\frac{2 b {\beta}^3}{c}+\frac{9 {\beta} c}{2}\right)^2}-\frac{9720 c^2}{{\beta}^4 \left(\frac{2 b {\beta}^3}{c}
+\frac{9 {\beta} c}{2}\right)^2}\right)\\
		& (-\frac{32508 b^3}{5 {\beta}^9 c^3 \left(\frac{{\beta} c^3}{4}-\frac{b c}{{\beta}}\right)^2}
+\frac{48519 b^2}{5 {\beta}^7 c \left(\frac{{\beta} c^3}{4}-\frac{b c}{{\beta}}\right)^2}+\frac{32508 b}
{5 {\beta}^8 \left(\frac{{\beta} c^3}{4}-\frac{b c}{{\beta}}\right)^3}-\frac{35991 b c}{5 {\beta}^5
\left(\frac{{\beta} c^3}{4}-\frac{b c}{{\beta}}\right)^2}\\
		&+\frac{3159 c^3}{8 {\beta}^3 \left(\frac{{\beta} c^3}{4}-\frac{b c}{{\beta}}\right)^2}-\frac{2 b}
{{\beta}^3 c}-\frac{9 c}{2 {\beta}})-\frac{72 b}{{\beta}^5 c^4}-\frac{36288 b}{{\beta}^7 c^2
\left(\frac{2 b {\beta}^3}{c}+\frac{9 {\beta} c}{2}\right)^2}-\frac{13464 b}{5 {\beta}^6 c^3 \left(\frac{2 b}{{\beta}^3 c}
+\frac{9 c}{2 {\beta}}\right)}\\
		&+\frac{3402}{{\beta}^5 \left(\frac{2 b {\beta}^3}{c}+\frac{9 {\beta} c}{2}\right)^2}-\frac{432}{{\beta}^4 c
 \left(\frac{2 b}{{\beta}^3 c}+\frac{9 c}{2 {\beta}}\right)}+\frac{18}{{\beta}^3 c^2}+O(\epsilon)\ne0.
	\end{split}
\end{equation*}
Now, we want to turn $\bar{j}_{20}$ and $\bar{j}_{31}$ into $1$ and
notice that the signs of the coefficients of $u_5^2$ and $u_5^3v_5$ change as the signs of $\bar{j}_{20}$ and $\bar{j}_{31}$.
The details are as follows.

(i) If $\bar{j}_{20}>0,\bar{j}_{31}>0$, then system (\ref{system4.8}) becomes the following form with the transformation
$(u_{4},v_{4})=(\bar{j}_{20}^{\frac{1}{5}}\bar{j}_{31}^{-\frac{2}{5}}u_{5},\bar{j}_{20}^{\frac{4}{5}}
\bar{j}_{31}^{-\frac{3}{5}}v_{5})$ and $\tau=\bar{j}_{20}^{-\frac{3}{5}}\bar{j}_{31}^{\frac{1}{5}}t$.
\begin{equation*}
	\begin{cases}
		\dot{u_{5}}=v_{5},\\
		\dot{v_{5}}=\bar{k}_{00}+\bar{k}_{10}u_{5}+\bar{k}_{01}v_{5}+u_{5}^{2}
		+\bar{k}_{11}u_{5}v_{5}+u_{5}^{3}v_{5}+R(u_{5},v_{5},\epsilon),
	\end{cases}
\end{equation*}
where
\begin{align*}
\bar{k}_{00}=\bar{j}_{00}{j}_{20}^{-\frac{7}{5}}\bar{j}_{31}^{\frac{4}{5}}\bar,\quad
\bar{k}_{10}=\bar{j}_{10}\bar{j}_{20}^{-\frac{6}{5}}\bar{j}_{31}^{\frac{2}{5}},\quad
\bar{k}_{01}=\bar{j}_{01}\bar{j}_{20}^{-\frac{3}{5}}\bar{j}_{31}^{\frac{1}{5}},\quad
\bar{k}_{11}=\bar{j}_{11}\bar{j}_{20}^{-\frac{2}{5}}\bar{j}_{31}^{-\frac{1}{5}}.
\end{align*}
Additionally, $R(u_{5},v_{5},\epsilon)$ has the same properties as $R(u_{3},v_{3},\epsilon)$.

(ii) If $\bar{j}_{20}<0,\bar{j}_{31}>0$ or $\bar{j}_{20}>0,\bar{j}_{31}<0$,
then system (\ref{system4.8}) becomes the following form with the transformation
$(u_{4},v_{4})=(\bar{j}_{20}^{\frac{1}{5}}\bar{j}_{31}^{-\frac{2}{5}}u_{5},-\bar{j}_{20}^{\frac{4}{5}}
\bar{j}_{31}^{-\frac{3}{5}}v_{5})$ and $\tau=-\bar{j}_{20}^{-\frac{3}{5}}\bar{j}_{31}^{\frac{1}{5}}t$.
\begin{equation*}
	\begin{cases}
		\dot{u_{5}}=v_{5},\\
		\dot{v_{5}}=\bar{k}_{00}+\bar{k}_{10}u_{5}+\bar{k}_{01}v_{5}+u_{5}^{2}
		+\bar{k}_{11}u_{5}v_{5}-u_{5}^{3}v_{5}+R(u_{5},v_{5},\epsilon),
	\end{cases}
\end{equation*}
where
\begin{align*}
\bar{k}_{00}=\bar{j}_{00}{j}_{20}^{-\frac{7}{5}}\bar{j}_{31}^{\frac{4}{5}}\bar,\quad
\bar{k}_{10}=\bar{j}_{10}\bar{j}_{20}^{-\frac{6}{5}}\bar{j}_{31}^{\frac{2}{5}},\quad
\bar{k}_{01}=-\bar{j}_{01}\bar{j}_{20}^{-\frac{3}{5}}\bar{j}_{31}^{\frac{1}{5}},\quad
\bar{k}_{11}=-\bar{j}_{11}\bar{j}_{20}^{-\frac{2}{5}}\bar{j}_{31}^{-\frac{1}{5}}.
\end{align*}

(iii) If $\bar{j}_{20}<0,\bar{j}_{31}<0$,
then system (\ref{system4.8}) becomes the following form with the transformation
$(u_{4},v_{4})=(-\bar{j}_{20}^{\frac{1}{5}}\bar{j}_{31}^{-\frac{2}{5}}u_{5},-\bar{j}_{20}^{\frac{4}{5}}
\bar{j}_{31}^{-\frac{3}{5}}v_{5})$ and $\tau=\bar{j}_{20}^{-\frac{3}{5}}\bar{j}_{31}^{\frac{1}{5}}t$.
\begin{equation*}
	\begin{cases}
		\dot{u_{5}}=v_{5},\\
		\dot{v_{5}}=\bar{k}_{00}+\bar{k}_{10}u_{5}+\bar{k}_{01}v_{5}-u_{5}^{2}
		+\bar{k}_{11}u_{5}v_{5}-u_{5}^{3}v_{5}+R(u_{5},v_{5},\epsilon),
	\end{cases}
\end{equation*}
where
\begin{align*}
\bar{k}_{00}=-\bar{j}_{00}{j}_{20}^{-\frac{7}{5}}\bar{j}_{31}^{\frac{4}{5}}\bar,\quad
\bar{k}_{10}=\bar{j}_{10}\bar{j}_{20}^{-\frac{6}{5}}\bar{j}_{31}^{\frac{2}{5}},\quad
\bar{k}_{01}=\bar{j}_{01}\bar{j}_{20}^{-\frac{3}{5}}\bar{j}_{31}^{\frac{1}{5}},\quad
\bar{k}_{11}=-\bar{j}_{11}\bar{j}_{20}^{-\frac{2}{5}}\bar{j}_{31}^{-\frac{1}{5}}.
\end{align*}

(F) Finally, we get the universal unfolding with the transformation
$(u_{6},v_{6})=(u_{5}-\frac{\bar{k}_{10}}{2},v_{5})$
\begin{equation}\label{system4.9}
	\begin{cases}
		\dot{u_{6}}=v_{6},\\
		\dot{v_{6}}=\bar{l}_{1}+\bar{l}_{2}v_{6}+\bar{l}_{3}u_{6}v_{6}
		+\bar{l}_{4}u_{6}^{2}+\bar{l}_{5}u_{6}^{3}v_{6}+R(u_{6},v_{6},\epsilon),
	\end{cases}
\end{equation}
where $R(u_{6},v_{6},\epsilon)$ has the same properties as $R(u_{3},v_{3},\epsilon)$.
There are three results corresponding to the three situations in (E).

(i) If $\bar{j}_{20}>0,\bar{j}_{31}>0$, then the cofficients of system (\ref{system4.9}) are
\begin{align*}
	\bar{l}_1&=\bar{k}_{00}-\frac{1}{4}\bar{k}_{10}^2,&
	\bar{l}_2&=\bar{k}_{01}-\frac{1}{8}\bar{k}_{10}^3-\frac{1}{2}\bar{k}_{10}\bar{k}_{11},&
	\bar{l}_3&=\bar{k}_{11}+\frac{3}{4}\bar{k}_{10}^2,&
	\bar{l}_{4}&=1,& \bar{l}_{5}=1.
\end{align*}

(ii) If $\bar{j}_{20}<0,\bar{j}_{31}>0$ or $\bar{j}_{20}>0,\bar{j}_{31}<0$,
then the cofficients of system (\ref{system4.9}) are
\begin{align*}
	\bar{l}_1&=\bar{k}_{00}-\frac{1}{4}\bar{k}_{10}^2,&
	\bar{l}_2&=\bar{k}_{01}+\frac{1}{8}\bar{k}_{10}^3-\frac{1}{2}\bar{k}_{10}\bar{k}_{11},&
	\bar{l}_3&=\bar{k}_{11}-\frac{3}{4}\bar{k}_{10}^2,&
	\bar{l}_{4}&=1,& \bar{l}_{5}&=-1.
\end{align*}

(iii) If $\bar{j}_{20}<0,\bar{j}_{31}<0$, then the cofficients of system (\ref{system4.9}) are
\begin{align*}
	\bar{l}_1&=\bar{k}_{00}+\frac{1}{4}\bar{k}_{10}^2,&
	\bar{l}_2&=\bar{k}_{01}-\frac{1}{8}\bar{k}_{10}^3+\frac{1}{2}\bar{k}_{10}\bar{k}_{11},&
	\bar{l}_3&=\bar{k}_{11}-\frac{3}{4}\bar{k}_{10}^2,&
	\bar{l}_{4}&=-1,&\bar{l}_{5}&=-1.
\end{align*}

Then with the help of the MATLAB, we can obtain
\begin{equation*}
	\left |
    \frac{\partial (\bar{l}_1,\bar{l}_2,\bar{l}_3 )}{\partial (\epsilon_1,\epsilon_2,\epsilon_3 )}
    \right |
	_{\epsilon_1=\epsilon_2=\epsilon_3=0} \ne0
\end{equation*}
for all three possible situations in (F).
Therefore, according to the theory in \cite{Wiggins,Dumortier}, system (\ref{system2.1}) undergoes the Bogdanov-Takens
bifurcation of codimension $3$ in a small neighborhood of $E_1$.
\end{proof}

The bifurcation diagram for system (\ref{system4.9}) can be described as follows.
If $l_1<0$, there are no equilibria; if $l_1=0$, then there is a saddle-node bifurcation plane in a small neighborhood of the origin
$(0,0)$; if $l_1>0$, then the system has two equilibria, a saddle and an anti-saddle.
The remaining surfaces of the bifurcation diagram in $\mathbb{R}^3$ have a conical structure, emanating from $(l_1,l_2,l_3)=(0,0,0)$,
 which can be demonstrated by drawing its intersection with the half sphere
\begin{equation*}
S=\left \{ (l_1,l_2,l_3)|l_1^2+l_2^2+l_3^2=\rho^2,l_1\ge 0,\rho>0 \ \text{sufficiently small}\right \}.
\end{equation*}
Now we project the traces onto the $l_2l_3$-plane for clear visualization.
\begin{figure}[!h]
  \begin{center}
   \includegraphics[width=0.5\textwidth]{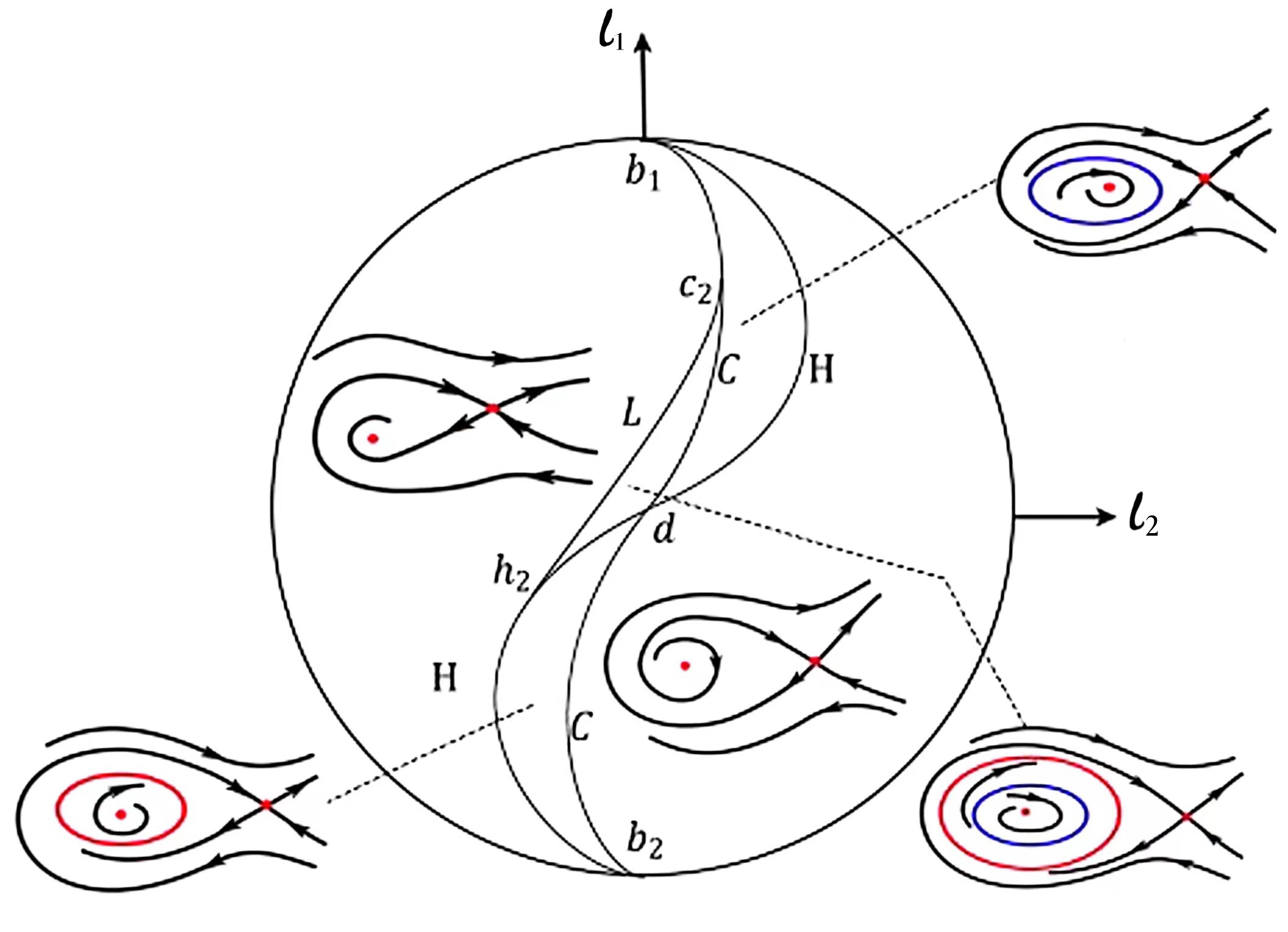}
  \end{center}
  \vskip
-10pt \caption{\small Bifurcation diagram of system (\ref{system4.9}) on $S$.}
\label{Figure5}
\end{figure}

Here we summarize the bifurcation in  system (\ref{system4.9}) based on the above discussion.
Figure 1 contains the Hopf bifurcation curve, the homoclinic bifurcation curve and the saddle-node bifurcation curve, which are represented by $H,C$,
and $\partial S$, respectively, where $\partial S$ is the boundary of $S$.
The curves $H$ and $C$ have the first order contact with the boundary of $S$ at the points $b_1$ and $b_2$.
The curve $L$ is tangent to the curve $C$ at point $c_2$ and tangent to the curve $H$ at point $h_2$, which is the saddle-node bifurcation
curve of double limit cycles.
The system (\ref{system4.9}) is a cusp singularity unfolding of codimension 2 around $b_1$ and $b_2$.

Along the $H$, when crossing the arc $b_1h_2$ of $H$ from right to left, the curve $H$ is a subercritical Hopf bifurcation with an unstable cycle
curve of codimension one. And the curve $H$ is a supercritical Hopf bifurcation with a stable cycle curve of codimension one when crossing the
arc $h_2b_2$ of $H$ from left to right. The Hopf bifurcation of codimension 2 occurs at point $h_2$.

A homoclinic bifurcation of codimension 1 occurs along the curve $C$.
When crossing the arc $b_1c_2$ of $C$ from left to right, the two parts of the saddle point coincide and an unstable limit cycle appears.
And the two parts of the saddle point coincide and a stable limit cycle appears when crossing  the arc $c_2b_2$ of $C$ from right to left.
A homoclinic bifurcation of codimension 2 occurs at point $c_2$.

Then we give some numerical simulations about the system.
In Figure 2, note that $E_0$ is a stable node when $c=0.1,\beta=0.12,b=0.08,d=0.08$. In Figure 3, there is a boundary equilibrium
point $E_0$ and a positive equilibrium point $E_1$. As $d=0.4$, $E_1$ is a cusp when $c=0.4,\beta=0.5477,b=0.012$;
$E_1$ is a saddle-node with the stable parabolic sector when $c=0.3,\beta=0.5,b=0.01$;
$E_1$ is a saddle-node with the unstable parabolic sector when $c=0.45,\beta=0.5,b=0.01$. There is a positive equilibrium $E_2$ in Figure 4.
Choose $\beta=0.6,b=0.0125$, $E_2$ is a center when $c=0.4,d=0.4$; $E_2$ is a source when $c=0.45,d=0.38$; $E_2$ is a sink when $c=0.3,d=0.5$.
$E_3$ is a saddle point when $c=0.3,\beta=0.5,b=0.0075,d=0.4$, which is shown in Figure 5.
\begin{figure}[!h]
  \begin{center}
   \includegraphics[width=0.45\textwidth]{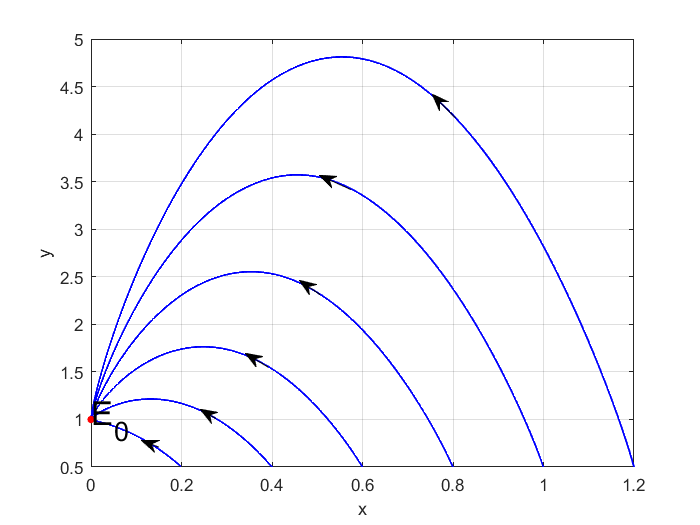}
  \end{center}
  \vskip
-10pt \caption{\small $E_0$ is a stable node when $c=0.1,\beta=0.12,b=0.08,d=0.08$.}
\label{Figure1}
\end{figure}

\begin{figure}[htbp]
    \centering
    \subfigure[]{
        \includegraphics[width=0.45\textwidth]{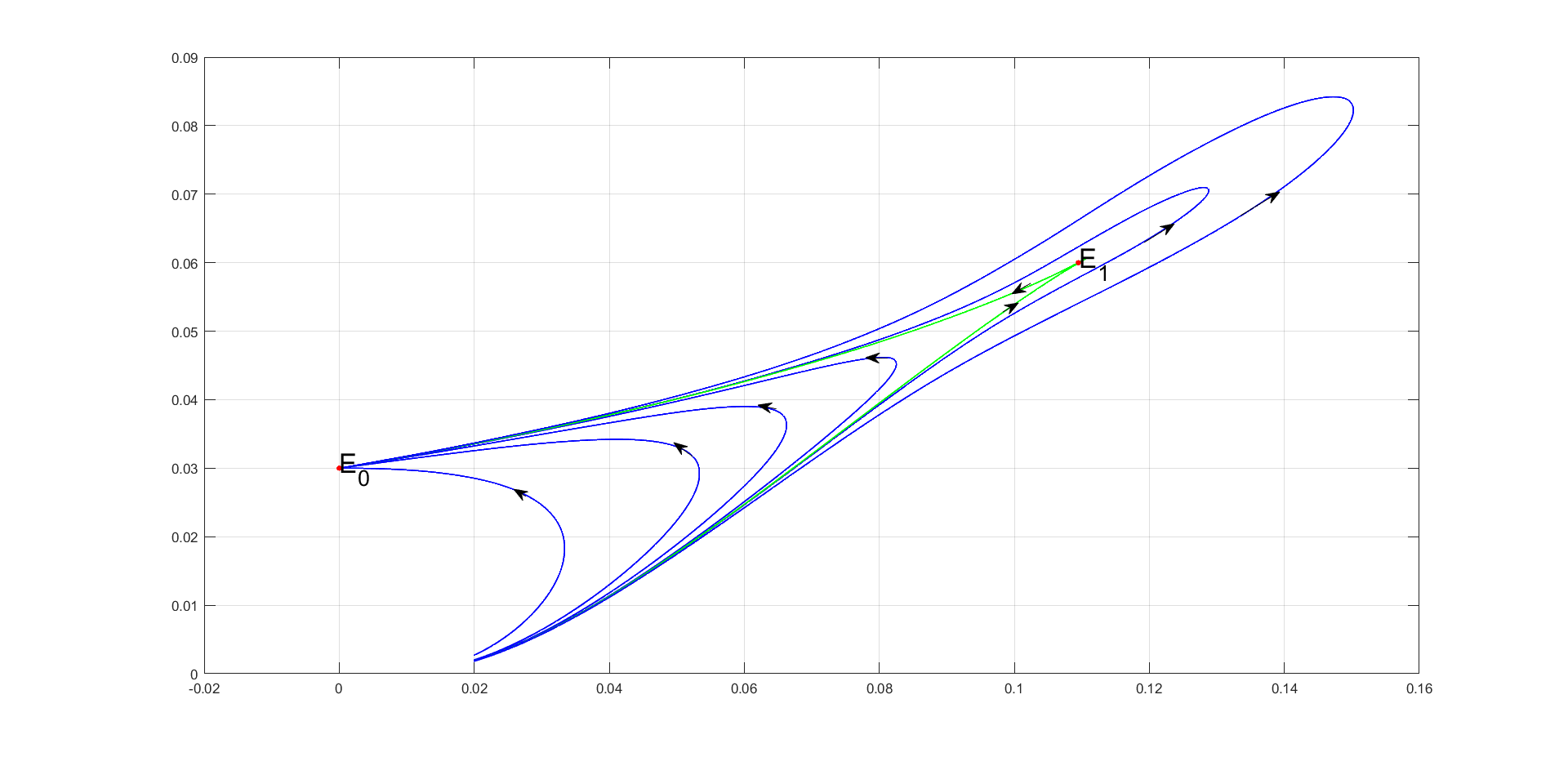}
        \label{fig:subfig1}
    }
    \subfigure[]{
        \includegraphics[width=0.45\textwidth]{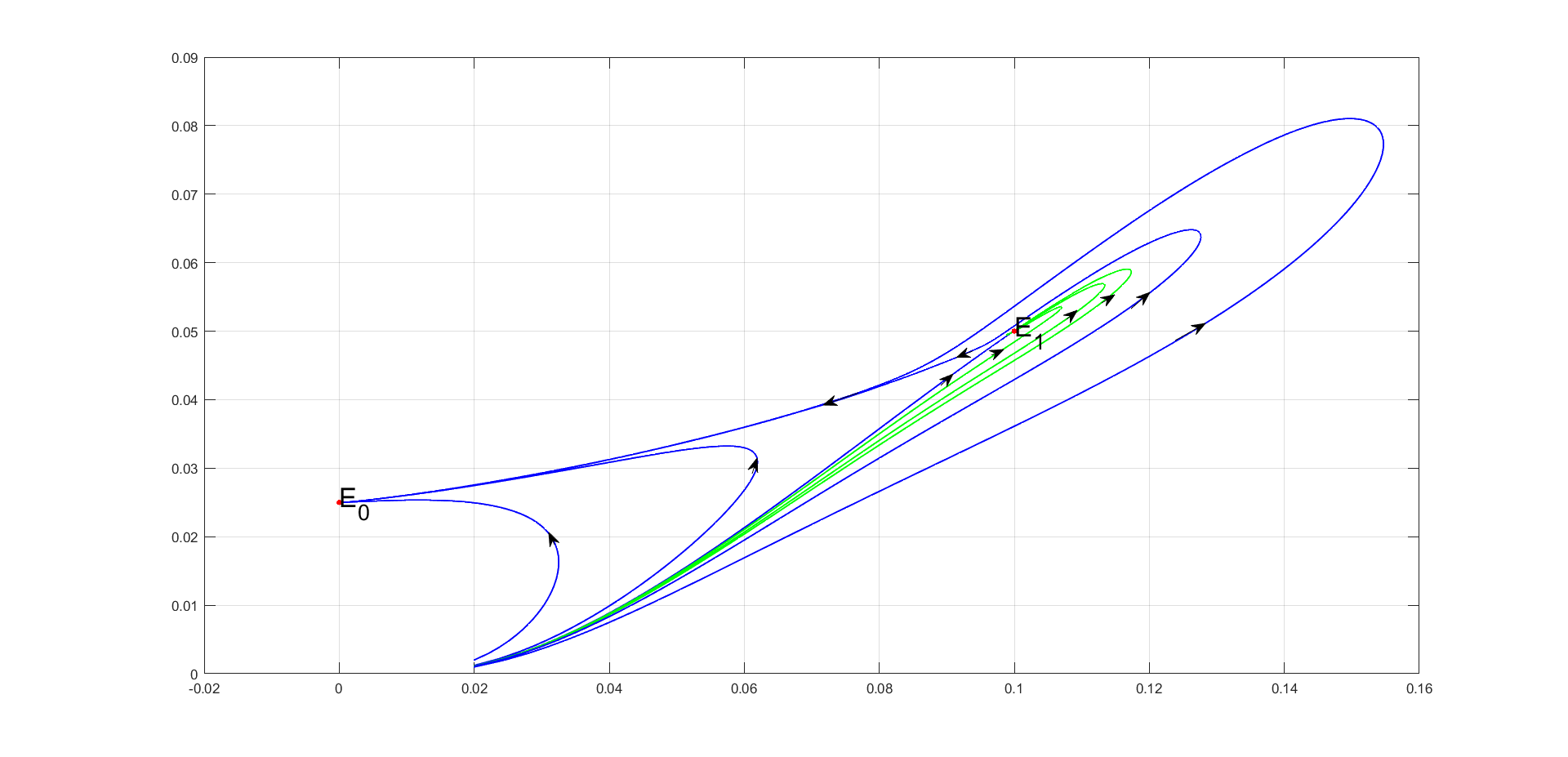}
        \label{fig:subfig2}
    }
    \subfigure[]{
        \includegraphics[width=0.45\textwidth]{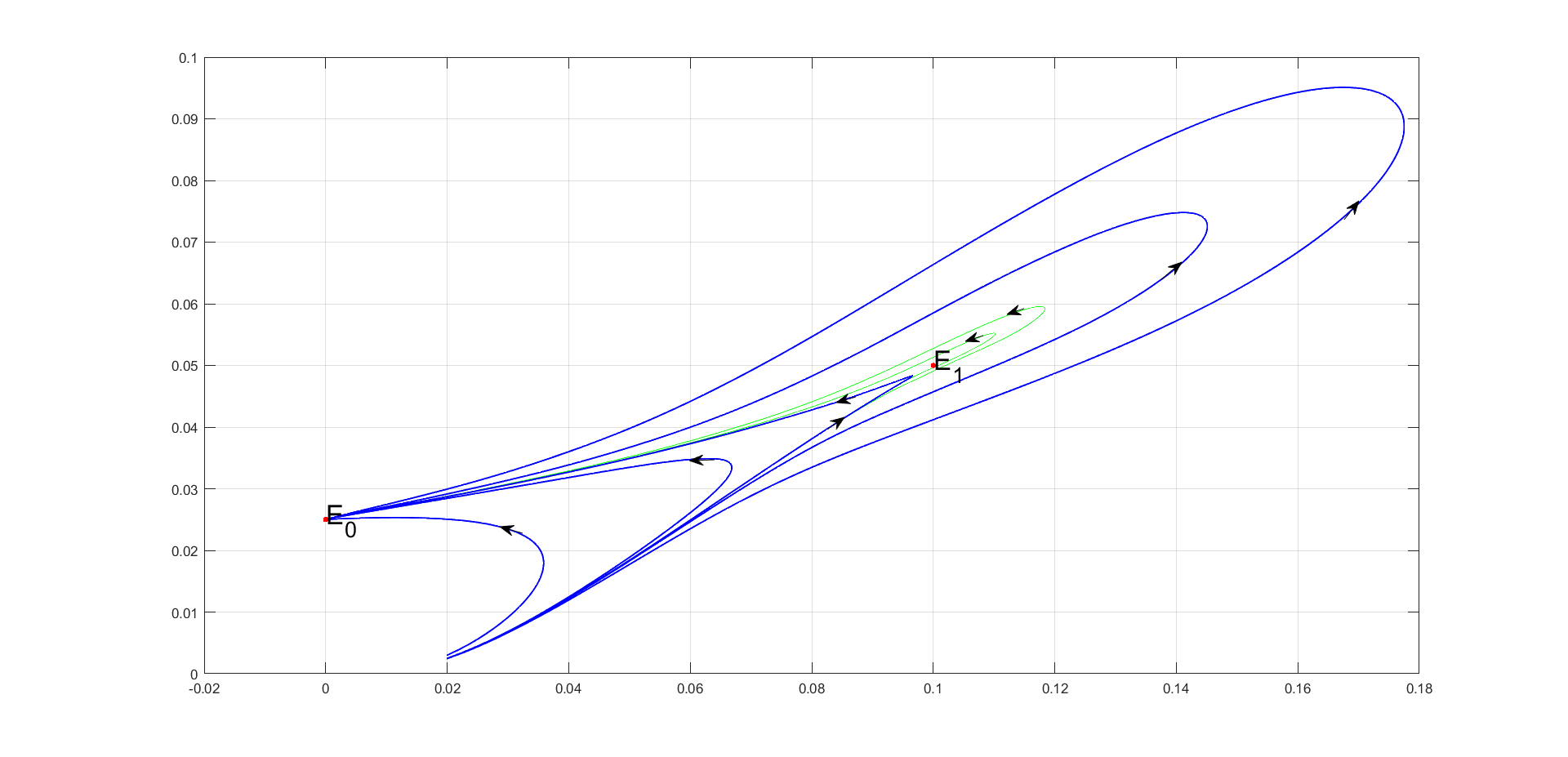}
        \label{fig:subfig3}
    }
    \caption{\small The stability of  $E_1$
    (a) $E_1$ is a cusp when $c=0.4,\beta=0.5477,b=0.012,d=0.4$
    (b) $E_1$ is a saddle-node with stable parabolic sector when $c=0.3,\beta=0.5,b=0.01,d=0.4$.
    (c) $E_1$ is a saddle-node with unstable parabolic sector when $c=0.45,\beta=0.5,b=0.01,d=0.4$.}
    \label{fig:all}
\end{figure}

\begin{figure}[htbp]
    \centering
    \subfigure[]{
        \includegraphics[width=0.45\textwidth]{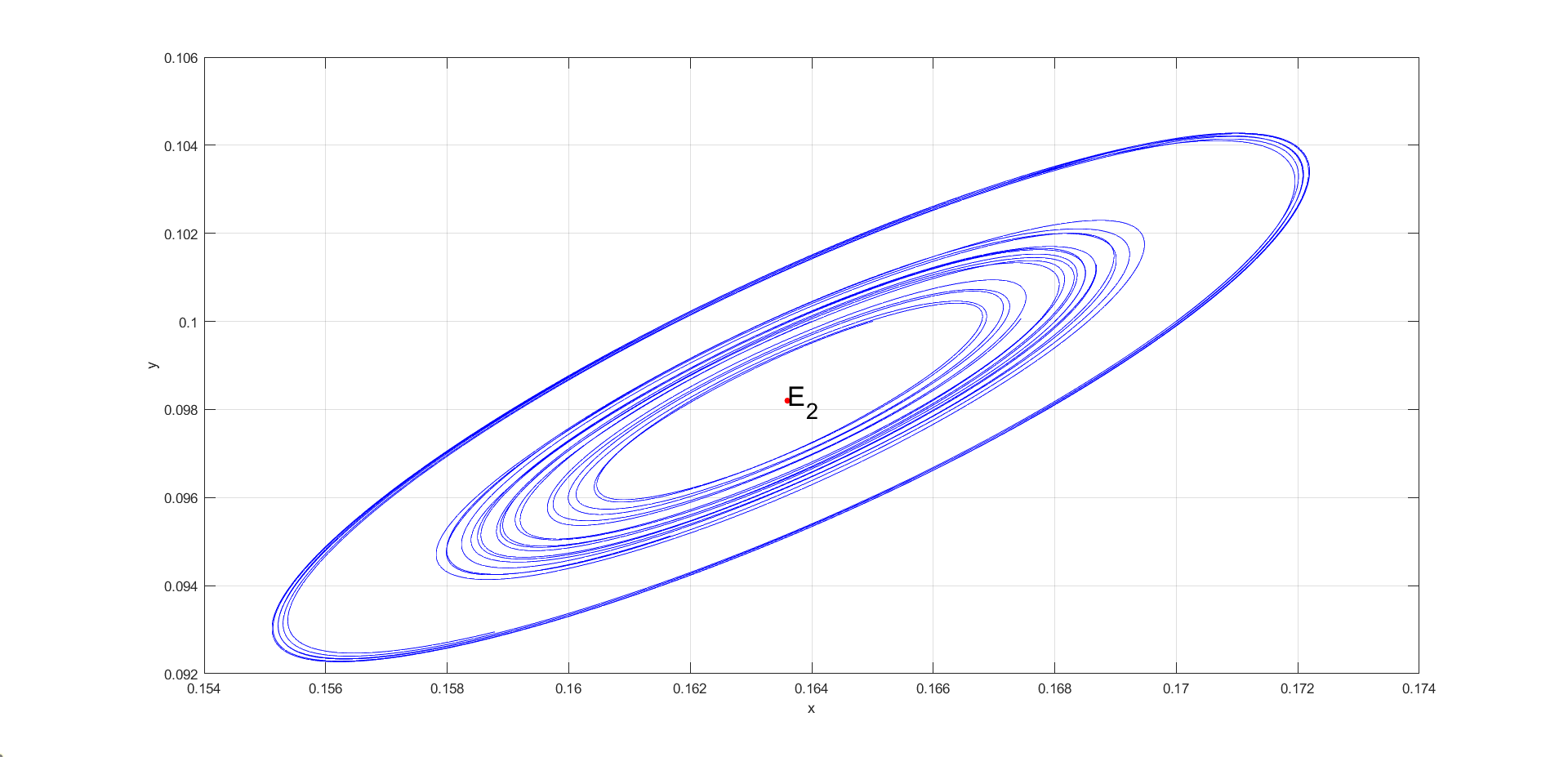}
        \label{fig:subfig1}
    }
    \subfigure[]{
        \includegraphics[width=0.45\textwidth]{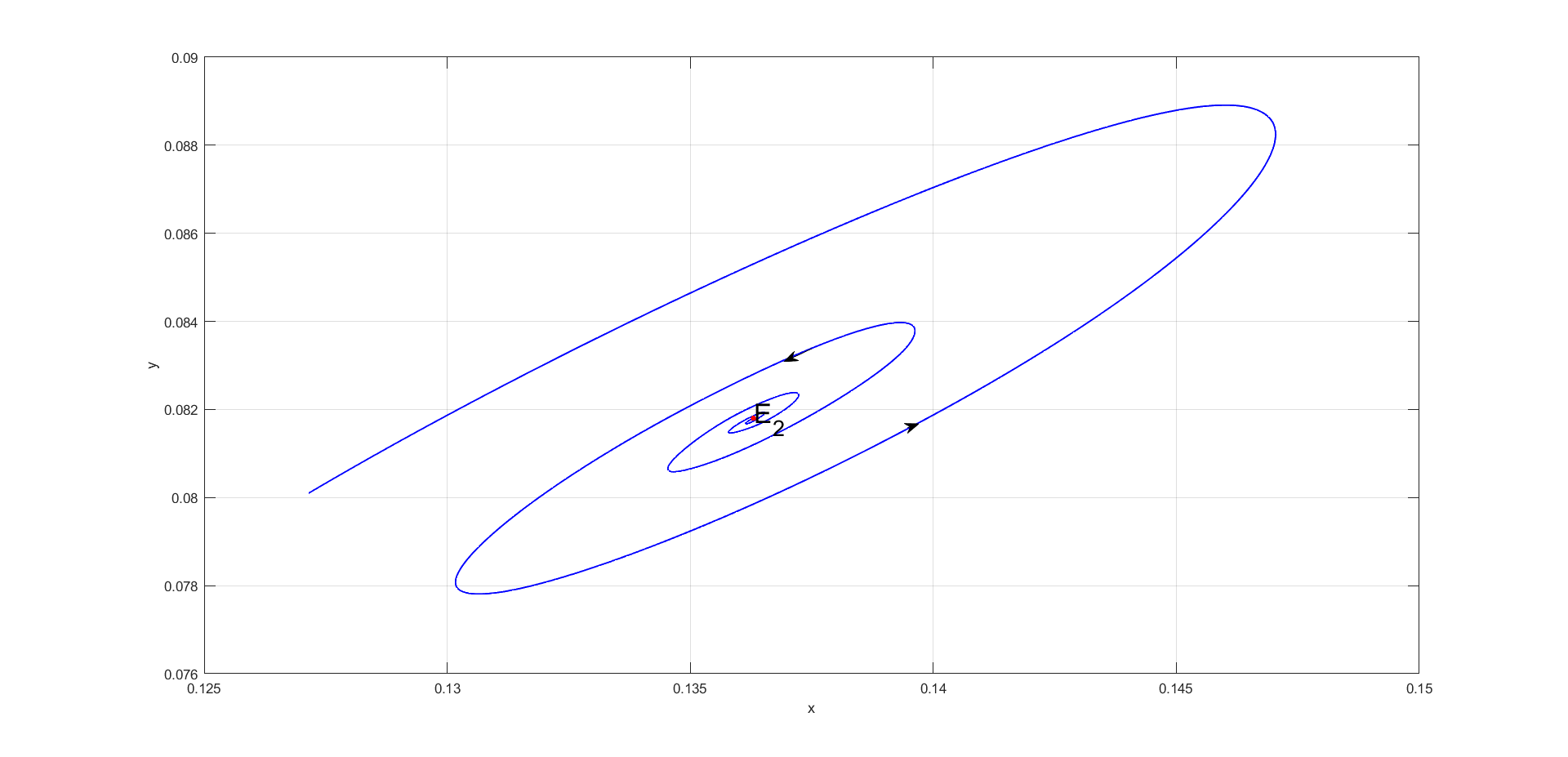}
        \label{fig:subfig2}
    }
    \subfigure[]{
        \includegraphics[width=0.45\textwidth]{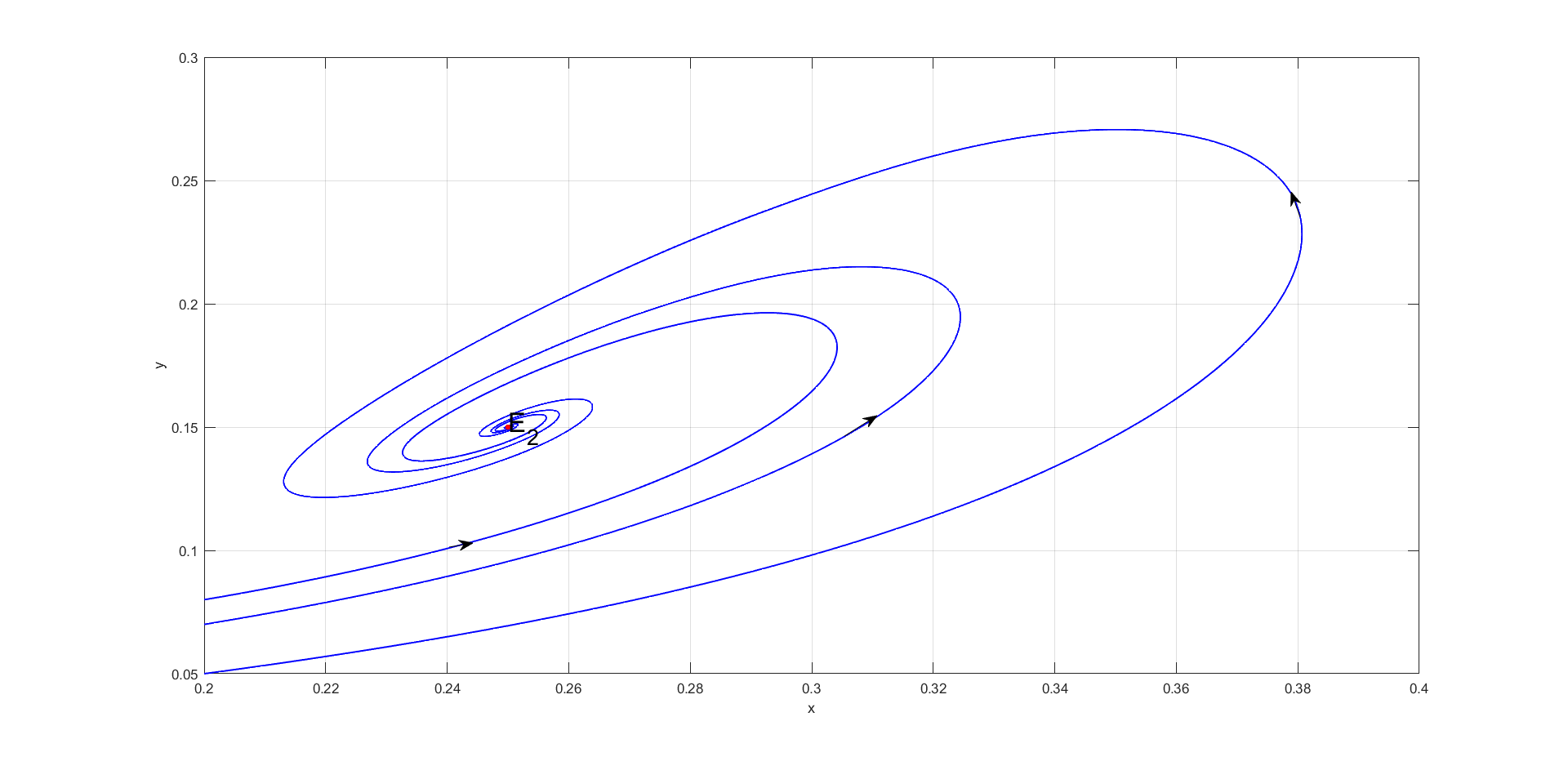}
        \label{fig:subfig3}
    }
    \caption{\small The stability of the $E_2$
    (a) $E_2$ is a center when $c=0.4,\beta=0.6,b=0.0125,d=0.4.$
    (b) $E_2$ is a source when $c=0.45,\beta=0.6,b=0.0125,d=0.38$.
    (b) $E_2$ is a sink when $c=0.3,\beta=0.6,b=0.0125,d=0.5$.}
    \label{fig:all}
\end{figure}

\begin{figure}[!h]
  \begin{center}
   \includegraphics[width=0.5\textwidth]{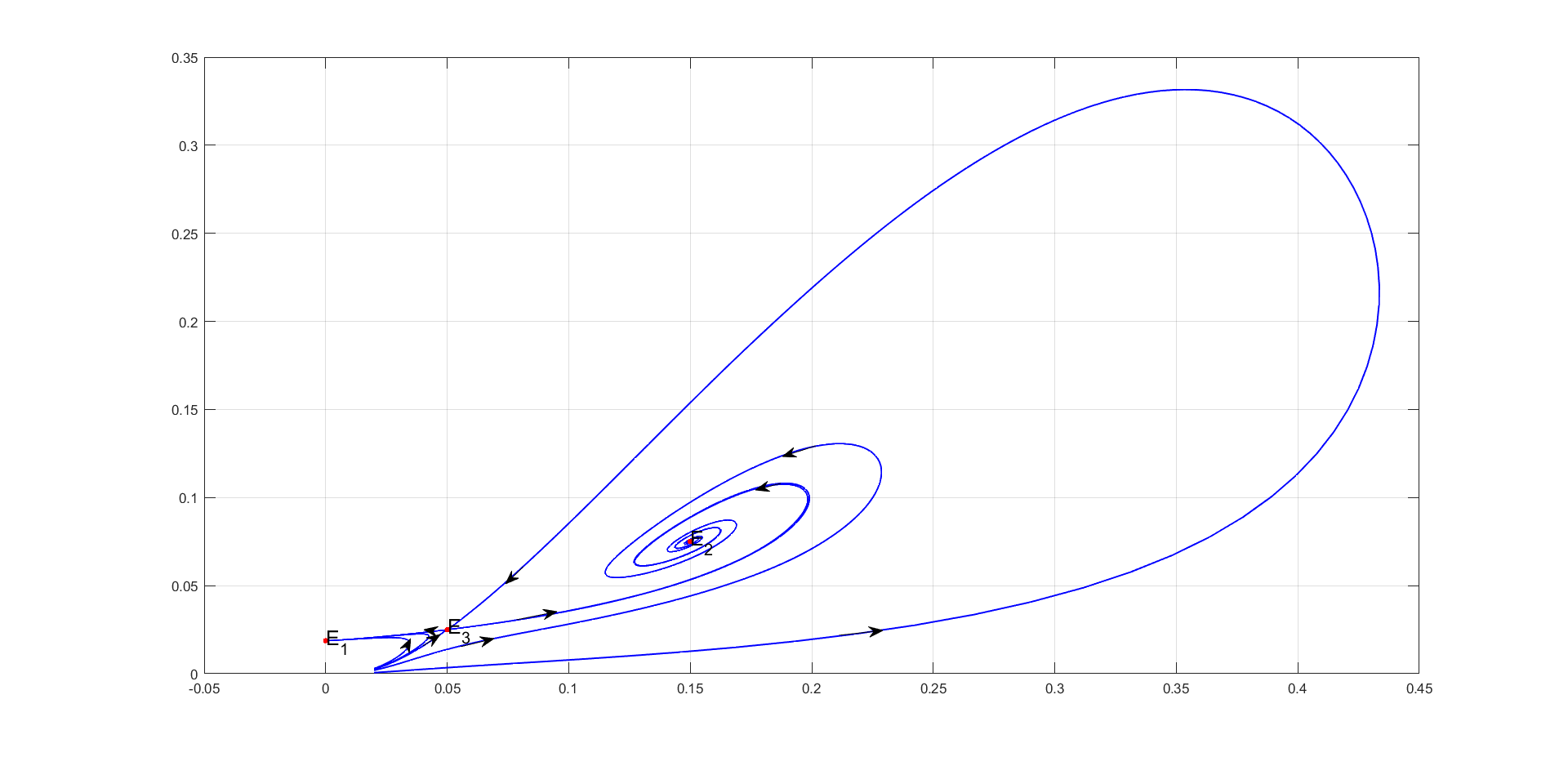}
  \end{center}
  \vskip
-10pt \caption{\small $E_3$ is a saddle point when $c=0.3,\beta=0.5,b=0.0075,d=0.4$.}
\label{Figure1}
\end{figure}

\section{Conclusion}\label{section-5}

Bifurcation analysis of the Gierer-Meinhardt model is carried out. Besides the saddle-node bifurcation and
the Hopf bifurcation, it is found that the degenerate Bogdanov-Takens bifurcation of codimension-3 appears in the model.
That was not reported in the previous results. By a series of transformation and based on the bifurcation theory, including
the Sotomayor's theorem and the normal form method, the  detailed bifurcation results are presented and more interesting dynamics
are revealed. Theoretical findings are verified in numerical simulation.  More further dynamics could be explored for the model.

\section*{Acknowledgments}
This work was supported by the National Natural Science Foundation of China (Nos. 11971032, 62073114).

\section*{Conflict of Interest}
The authors declare that there are no conflicts of interest.

\section*{Data availability statement}

All data generated or analysed during this study are included in this published article or available upon request.

\end{document}